%% file: equiv_kh.tex
\documentclass[a4paper,10pt]{amsart}
\usepackage[utf8]{inputenc}
\usepackage[T1]{fontenc}
\usepackage[english]{babel}
\usepackage{hyperref}
\usepackage{todonotes}

\usepackage{amsmath}
\usepackage{amsthm}
\usepackage{amssymb}
\usepackage{amsfonts}
\usepackage[alphabetic,abbrev]{amsrefs}

\usepackage[pagewise,displaymath]{lineno}

\usepackage{import}

\usepackage{tikz}
\usetikzlibrary{matrix, arrows, tqft}

\usepackage{calc}

\newtheorem{thm}{Theorem}[section]
\newtheorem{prop}[thm]{Proposition}
\newtheorem{lemma}[thm]{Lemma}
\newtheorem{cor}[thm]{Corollary}
\newtheorem{rem}[thm]{Remark}

\theoremstyle{definition}
\newtheorem{defn}[thm]{Definition}
\newtheorem{ex}[thm]{Example}

{%
\vskip 6pt%
\noindent%
\textbf{Theorem #1. }%
\ignorespaces%
}%
{%
\vskip 6pt%
\par\noindent\ignorespacesafterend%
}

{%
\vskip 6pt%
\noindent%
\textbf{Proposition #1. }%
\ignorespaces%
}%
{%
\vskip 6pt%
\par\noindent\ignorespacesafterend%
}

{%
\vskip 6pt%
\noindent%
\textbf{Corollary #1. }%
\ignorespaces%
}%
{%
\vskip 6pt%
\par\noindent\ignorespacesafterend%
}

\newcommand{\z}[1]{\mathbb{Z}_{#1}}
\newcommand{\zzz}{\mathbb{Z}}
\newcommand{\qqq}{\mathbb{Q}}
\newcommand{\fff}{\mathbb{F}}
\newcommand{\rrr}{\mathbb{R}}

\newcommand{\sign}{\operatorname{sign}}

\newcommand{\Vector}[2]{\left[ \begin{array}{r} #1 \\ #2 \end{array} \right]}

\newcommand{\supp}{\operatorname{supp}}

\newcommand{\GpAlg}[2]{\mathbb{#1}\left[ #2 \right]}
\newcommand{\GpAAlg}[2]{#1 \left[ #2 \right]}
\newcommand{\gpAlg}[1]{\GpAlg{Q}{#1}}

\newcommand{\gpRing}[1]{\GpAlg{Z}{#1}}
\newcommand{\cyclicGroupAlg}[1]{\gpAlg{\z{#1}}}
\newcommand{\cyclicGroupRing}[1]{\gpRing{\z{#1}}}
\newcommand{\cyclotomicRing}[1]{\gpRing{\xi_{#1}}}
\newcommand{\cyclotomicField}[1]{\gpAlg{\xi_{#1}}}

\newcommand{\ind}[2]{\operatorname{Ind}^{#1}_{#2}}
\newcommand{\res}[2]{\operatorname{Res}^{#1}_{#2}}

\newcommand{\Hom}[4]{\operatorname{Hom}^{#1}_{#2} \left(#3, #4 \right)}
\newcommand{\Ext}[4]{\operatorname{\mathbb{E}xt}^{#1}_{#2} \left(#3, #4\right)}

\newcommand{\ext}[4]{\operatorname{Ext}^{#1}_{#2} \left(#3, #4\right)}

\newcommand{\cone}{\operatorname{Cone}}
\newcommand{\Tot}{\operatorname{Tot}}

\newcommand{\cross}{\operatorname{Cr}}
\newcommand{\nonorientCross}{\raisebox{-5pt}{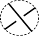}}
\newcommand{\zeroSmoothing}{\raisebox{-5pt}{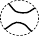}}
\newcommand{\oneSmoothing}{\raisebox{-5pt}{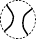}}
\newcommand{\posCross}{\raisebox{-5pt}{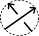}}
\newcommand{\negCross}{\raisebox{-5pt}{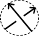}}
\newcommand{\orientRes}{\raisebox{-5pt}{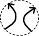}}
\newcommand{\nonorientRes}{\raisebox{-5pt}{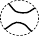}}
\newcommand{\nonorientOrbit}{\raisebox{-5pt}{\import{./}{crossing_nonoriented.pdf_tex}} \ldots \raisebox{-5pt}{\import{./}{crossing_nonoriented.pdf_tex}}}
\newcommand{\posOrbit}{\raisebox{-5pt}{\import{./}{crossing_pos.pdf_tex}} \ldots \raisebox{-5pt}{\import{./}{crossing_pos.pdf_tex}}}
\newcommand{\negOrbit}{\raisebox{-5pt}{\import{./}{crossing_neg.pdf_tex}} \ldots \raisebox{-5pt}{\import{./}{crossing_neg.pdf_tex}}}
\newcommand{\orientResOrbit}{\raisebox{-5pt}{\import{./}{orient_resolution.pdf_tex}} \ldots \raisebox{-5pt}{\import{./}{orient_resolution.pdf_tex}}}
\newcommand{\nonorientResOrbit}{\raisebox{-5pt}{\import{./}{nonorient_resolution.pdf_tex}} \ldots \raisebox{-5pt}{\import{./}{nonorient_resolution.pdf_tex}}}
\newcommand{\zeroSmoothingOrbit}{\raisebox{-5pt}{\import{./}{0-smoothing.pdf_tex}} \ldots \raisebox{-5pt}{\import{./}{0-smoothing.pdf_tex}}}
\newcommand{\oneSmoothingOrbit}{\raisebox{-5pt}{\import{./}{1-smoothing.pdf_tex}} \ldots \raisebox{-5pt}{\import{./}{1-smoothing.pdf_tex}}}

\newcommand{\ccc}{\overline{\operatorname{CKh}}}
\newcommand{\cc}{\operatorname{CKh}}

\newcommand{\kh}{\operatorname{Kh}}
\newcommand{\equivKh}[4]{\kh_{\z{#1}}^{#2,#3,#4}}
\newcommand{\khp}{\operatorname{KhP}}

\newcommand{\bnbracket}[1]{[\kern-1.5pt [ #1 ]\kern-1.5pt]_{\operatorname{BN}}}
\newcommand{\khbracket}[1]{[\kern-1.5pt [ #1 ]\kern-1.5pt]_{\operatorname{Kh}}}

\newcommand{\One}{\operatorname{1 \kern-3.75pt 1}}

\newcommand{\StacksTag}[2]{\href{http://stacks.math.columbia.edu/tag/#1}{#2}}

\title{Equivariant Khovanov homology of periodic links}
\author{Wojciech Politarczyk}
\address{Department of Mathematics and Computer Science \\
  Adam Mickiewicz University in Pozna\'n \\
ul. Umultowska 87, 61-614 Pozna\'n, Poland}
\email{politarw@amu.edu.pl}

\subjclass[2010]{primary: 57M27, secondary: 57M25, 57M60, 55N91, 18G40}

\begin{document}
\maketitle

\begin{abstract}
  The purpose of this paper is to construct and study equivariant Khovanov homology – a version of Khovanov homology theory for periodic links.
  Since our construction works regardless of the characteristic of the coefficient ring it generalizes a previous construction by Chbili.
  We establish invariance under equivariant isotopies of links and study algebraic properties of integral and rational version of the homology theory.
  Moreover, we construct a skein spectral sequence converging to equivariant Khovanov homology and use this spectral sequence to compute, as an example, equivariant Khovanov homology of torus links \(T(n,2)\).
\end{abstract}


\section{Introduction}
A link $L$ is \textit{$m$-periodic} if it is invariant under a semi-free action of a cyclic group $\z{m}$ of order \(m\) on the $3$-sphere, i.e., for a fixed link $L$ we are interested in 
diffeomorphisms $f \colon (S^{3},L) \to (S^{3},L)$ of finite order, with the property that $L$ is 
disjoint from the fixed point set of $f$. Due to the resolution of the Smith 
Conjecture, see~\cite{BassMorgan},
existence of such symmetry can be rephrased in the following 
way. Let $\rho_{n}$ be the rotation of $\mathbb{R}^{3}$ by the $\frac{2\pi}{n}$ angle about the $OZ$ 
axis. We are interested in links $L \subset \mathbb{R}^{3}$, which are disjoint from the $OZ$ axis 
and invariant under $\rho_{n}$.

In recent years there has been an interest in studying periodic links with the aid of homological links invariant, see~\cite{borodzik_khovanov_2017,Chbili2,cornish_sutured_2016,Hendricks1,hendricks_flexible_2016,politarczyk_equivariant_2017,SeidelSmith1,zhang_rank_2017}. Many of these papers build upon a previous work on polynomials invariants of periodic links~\cite{Murasugi2,murasugi_jones_1988,Przytycki,Traczyk2,Traczyk,Traczyk3}.

The purpose of this paper is is to construct and study a homological invariant of periodic links which we call \emph{equivariant Khovanov homology}. The construction we present is an algebraic analogue of the Borel equivariant cohomology – one of the standard tools in equivariant algebraic topology, see~\cite{tomDieck}. The paper is an abridged version of the author's Ph.D. Thesis~\cite{politarczyk_wojciech_khovanov_2015}.

Firstly, we establish invariance of equivariant Khovanov homology i.e. we prove that it is invariant under equivariant isotopies of periodic links. Our construction generalizes the one from~\cite{Chbili2} because it works regardless of the characteristic of the coefficient ring. Moreover, in~\cite{borodzik_khovanov_2017,politarczyk_equivariant_2017} it was shown that this invariant can be effectively applied to decide whether a link is periodic or not.

Secondly, we construct a \textit{skein spectral sequence} for equivariant Khovanov homology. The skein spectral sequence, in order to preserve the symmetry of the considered diagram, takes into account not a single crossing but an orbit of crossings and all its resolutions.
As an application we compute equivariant Khovanov homology of torus links \(T(n,2)\) over the rational field.

The paper is organized as follows. In Section~\ref{sec:kh-bracket-periodic-links} we study the Bar-Natan's bracket of a periodic link diagram. We exhibit an action of the cyclic group on it and show that it descends to an action on any link homology theory obtained from this complex via TQFT. Moreover, we study equivariant isotopies of periodic diagrams: we define equivariant Reidemeister moves and equivariant version of Reidemeister theorem. In Section~\ref{sec: equiv kh} we give a construction of equivariant Khovanov homology, and prove its invariance under equivariant Reidemeister moves. Section~\ref{subsec: spectral sequence} is devoted to the construction of the skein spectral sequence.
Section~\ref{sec:computations} contains computations of equivariant Khovanov homology of torus links \(T(n,2)\).

\subsection*{Acknowledgements}
I would like to express my gratitude to my Ph.D. advisor Prof. Krzysztof Pawałowski for his guidance and support.
I am also grateful to Józef Przytycki for suggesting me to take up this project and to Maciej Borodzik for many suggestions and corrections which improved the paper.

\section{Khovanov bracket of a periodic links}
\label{sec:kh-bracket-periodic-links}

Let us first recall the definition of a periodic link.
\begin{defn}
Let $m > 1$ be an integer, and let $L$ be a link in $S^{3}$. We say that $L$ is \emph{$m$-periodic}, if there exists an action of the cyclic group \(\z{m}\) of order $m$ on $S^3$ satisfying the following conditions
\begin{enumerate}
\item The fixed point set, denoted by $F$, is an unknot,
\item $L$ is disjoint from $F$,
\item $L$ is a $\z{m}$-invariant subset of $S^3$.
\end{enumerate}
Analogously, we say that a link diagram \(\mathcal{D} \subset \rrr^2\) is \emph{\(m\)-periodic} if it is invariant under a rotation of \(\rrr^2\) of order \(m\) and it is disjoint from the center of this rotation.
\end{defn}

\begin{figure}
  \centering
  \def\svgscale{0.5}
  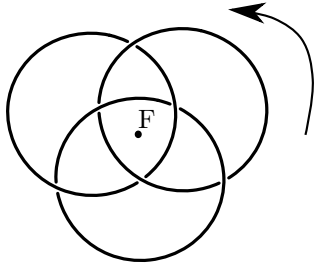
  \caption{Borromean rings are $3$-periodic. The fixed point axis $F$ is marked with a dot.}
\label{fig: bor rings 3-periodic}
\end{figure}

\begin{ex}
  As an example consider Borromean rings. This link is $3$-periodic. The symmetry is visualized on Figure 
  \ref{fig: bor rings 3-periodic}. The dot marks the fixed point axis.
\end{ex}

\begin{ex}\label{ex: periodic torus links}
  Torus links constitute an infinite family of periodic links. In fact, according to \cite{Murasugi2},
  the torus link $T(m,n)$ is $d$-periodic if, and only if, $d$ divides either $m$ or $n$.
\end{ex}

In the remainder part of this section we will study \emph{Khovanov bracket}, denoted by \(\khbracket{\mathcal{D}}\), of a periodic link diagram~\(\mathcal{D}\). We will assume that the reader is familiar with this concept. Moreover, we will use the language of \emph{planar algebras}. We refer the reader who is unfamiliar with these concepts to~\cite{BarNatan}.

For a link diagram \(\mathcal{D}\) we will denote by \(n(\mathcal{D})\), \(n_{+}(\mathcal{D})\) and \(n_{-}(\mathcal{D})\) the number of all, positive and negative crossings of \(\mathcal{D}\), respectively.

Periodic diagrams of periodic links can be described conveniently in terms of planar algebras.
Let us denote by \(\mathcal{D}_m\) the $m$-periodic \emph{planar arc diagram} with $m$~input disks which does not contain any closed arcs, see Figure~\ref{fig: planar diagram periodic} for an example with \(m=4\).
Choose a tangle diagram $\mathcal{T}$ with sufficient number of endpoints and glue $m$ copies of $\mathcal{T}$ into the input disks of $\mathcal{D}_{m}$. In this way, we obtain a periodic link diagram, denoted by $\mathcal{D}_{m}(\mathcal{T})$, quotient of which is represented by an appropriate closure of $\mathcal{T}$. See Figure \ref{fig: planar diagram periodic link} for an example.

\begin{figure}
\centering
\def\svgscale{0.6}  
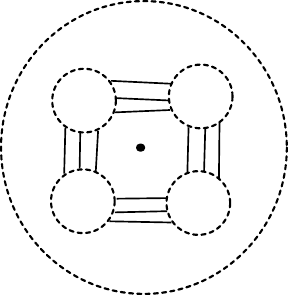
\caption{$4$-periodic planar diagram.}
\label{fig: planar diagram periodic}
\end{figure}

\begin{figure}
\centering
\def\svgscale{0.6}
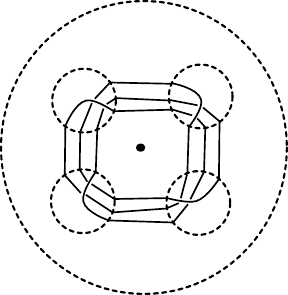
\caption{Torus knot $T(3,4)$ as a $4$-periodic knot obtained from the planar diagram from Figure 
\ref{fig: planar diagram periodic}}
\label{fig: planar diagram periodic link}
\end{figure}
Using the above description of periodic links, it is possible to exhibit a cobordism that induces an action of $\z{m}$ on the Khovanov bracket $\khbracket{\mathcal{D}}$, for $\mathcal{D}$ a periodic link diagram.
Firstly, notice that we can assume that $\mathcal{D}$ represents a link in $D^{2} \times I$ and \(\z{m}\) rotates the $D^{2}$ factor, where $D^{2}$ denotes a $2$-disk.
Secondly, note that the diffeomorphism, denote it by $f$, generating the $\z{m}$-symmetry of $D^{2} \times I$, is isotopic to the identity.
This isotopy can be chosen in such a way that it changes the angle of rotation linearly from $0$ to $\frac{2\pi}{m}$.
Denote this isotopy by $H$.
The cobordism in question is the trace of~$H$
\begin{linenomath*}
$$\Sigma_{H} = \{(H(x, t), t) \in D^{2} \times I \times I \colon x \in L, \quad t \in I \}.$$  
\end{linenomath*}
Topologically, $\Sigma_{H}$ is an embedded cylinder $S^{1} \times I$, however it is not isotopic, rel~boundary, to the trace of the constant isotopy $\Sigma_{H_{0}}$.

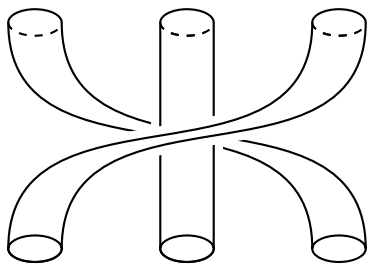
\begin{figure}
\centering
\begin{tikzpicture}
\begin{scope}[tqft/cobordism/.style={draw, thick},
tqft/every incoming lower boundary component/.style={draw, thick, dashed},
tqft/every outgoing lower boundary component/.style={draw, thick}]
\pic[tqft,
    at = {(0,0)},
    incoming boundary components = 1,
    outgoing boundary components = 1,
    offset = 2,
    genus = 0,
    cobordism height = 3cm];
\begin{scope}[tqft/cobordism outer path/.style={draw, white, line width=7pt}]
\pic[tqft/cylinder, at = {(2,0)}, cobordism height = 3cm];
\end{scope}
\pic[tqft/cylinder, at = {(2,0)}, cobordism height = 3cm];
\begin{scope}[tqft/cobordism outer path/.style={draw, white, line width=7pt}]
\pic[tqft,
    at = {(4,0)},
    incoming boundary components = 1,
    outgoing boundary components = 1,
    offset = -2,
    genus = 0,
    cobordism height = 3cm];
\end{scope}
\pic[tqft,
    at = {(4,0)},
    incoming boundary components = 1,
    outgoing boundary components = 1,
    offset = -2,
    genus = 0,
    cobordism height = 3cm];
\end{scope}
\end{tikzpicture}
\caption{Periodic Kauffman state with $3$ components and symmetry of order $2$. Middle cylinder 
contains the fixed point axis $F$.}
\label{fig: periodic Kauffman state}
\end{figure}

Let us now proceed to the description of the map induced by $\Sigma_{H}$ on \(\khbracket{\mathcal{D}}\).
First, we need to number crossings of $\mathcal{D}$.
Next, consider two vector spaces: $W$~which is spanned by the crossings of $\mathcal{D}$ and its exterior algebra $V = \Lambda^{\ast} W$.
Define the \textit{distinguished basis} of $V$ to be the basis consisting of vectors of the form
\begin{linenomath*}
$$c_{i_{1}} \wedge c_{i_{2}} \wedge \ldots \wedge c_{i_{k}}, \quad i_{1} < i_{2} < \ldots < i_{k}$$  
\end{linenomath*}
where $c_{i_{j}} \in \cross{\mathcal{D}}$ is a crossing of $\mathcal{D}$.
Notice that to each element $v$ of the distinguished basis we can associate a unique smoothing $\mathcal{D}_{v}$ of \(\mathcal{D}\) in the following way.
Let
\begin{linenomath*}
$$v = c_{i_1} \wedge c_{i_2} \wedge \ldots \wedge c_{i_k}, \quad i_1 < i_2 < \ldots < i_k,$$  
\end{linenomath*}
then $\mathcal{D}_{v}$ is obtained from $\mathcal{D}$ by smoothing crossings $c_{i_1}, c_{i_{2}}, \ldots, c_{i_{k}}$ with $1$-smoothing and the remainder crossings with $0$-smoothing, see Figure \ref{fig:smoothings}.

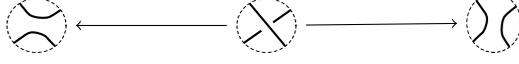
\begin{figure}
  \centering
  \begin{tikzpicture}
    \matrix(m)[matrix of math nodes, column sep=2cm] {
      \def\svgscale{1.2}\raisebox{-5pt}{\import{./}{0-smoothing.pdf_tex}} & \def\svgscale{1.2}\raisebox{-5pt}{\import{./}{crossing_nonoriented.pdf_tex}}{} & \def\svgscale{1.2}\raisebox{-5pt}{\import{./}{1-smoothing.pdf_tex}} \\
    };
    \path[->]
    (m-1-2) edge node {} (m-1-1)
    (m-1-2) edge node {} (m-1-3);
  \end{tikzpicture}
  \caption{$0$-smoothing and $1$-smoothing}
  \label{fig:smoothings}
\end{figure}

Let $\mathcal{T}$ be a tangle diagram and let $\mathcal{D} = \mathcal{D}_{m}(\mathcal{T})$.
Denote by $W_{\mathcal{T}}$ and $W_{\mathcal{D}}$ the vector space spanned by crossings of \(\mathcal{T}\) and \(\mathcal{D}\), respectively.
Under these assumptions 
\begin{linenomath*}
$$W_{\mathcal{D}} \cong W_{\mathcal{T}}^{m}, \quad \mbox{ and } \quad \Lambda^{\ast} W_{\mathcal{D}} \cong \left(\Lambda^{\ast} 
W_{\mathcal{T}}\right)^{\otimes m}.$$  
\end{linenomath*}
Symmetry of $\mathcal{D}$ induces an action of $\z{m}$ on $\Lambda^{\ast} W_{\mathcal{D}}$ which permutes factors in the tensor product above.
Cobordism $\Sigma_{H}$ discussed above induces a map
\begin{linenomath*}
$$\Sigma_{H} \colon \khbracket{\mathcal{D}} \to \khbracket{\mathcal{D}}$$  
\end{linenomath*}
which permutes Kauffman states of $\mathcal{D}$.
This permutation is compatible with the induced action on $\Lambda^{\ast} W_{\mathcal{D}}$.
Geometrically, the map $\Sigma_{H}|_{\khbracket{\mathcal{D}}^{r}}$ is induced, up to sign, by a ``permutation'' cobordisms similar to the one in Figure~\ref{fig: periodic Kauffman state}.

Let us define a map
\begin{linenomath*}
\begin{align}\label{eqn: permutation map}
\psi \colon \left(\Lambda^{\ast} W_{\mathcal{T}}\right)^{\otimes m} &\to \left(\Lambda^{\ast} W_{\mathcal{T}}\right)^{\otimes m} \\
\psi \colon x_{1} \otimes x_{2} \otimes \ldots \otimes x_{m} &\mapsto (-1)^{\alpha} x_{2} \otimes \ldots \otimes x_{m} \otimes x_{1}, \quad x_{i} \in W_{\mathcal{T}}, \nonumber
\end{align}  
\end{linenomath*}
where
\begin{linenomath*}
$$\alpha = (m-1) n_{-}(\mathcal{T}) + \deg x_{1} (\deg x_{2} + \deg x_{3} + \ldots + \deg x_{m}).$$  
\end{linenomath*}
Above, $n_{-}(\mathcal{T})$ stands for the number of negative crossings of $\mathcal{T}$.
Automorphism $\psi$ maps any vector from the distinguished basis of $\Lambda^{\ast} W_{\mathcal{D}}$ to $\pm 1$ multiplicity of some other vector from the basis.
\begin{linenomath*}
$$\psi(v) = \sign(\psi,v) w$$  
\end{linenomath*}
We can use these signs to change the definition of $\Sigma_{H}$ as follows.
\begin{linenomath*}
\begin{align}
\Sigma_{H}|_{D_{v}} &\colon \mathcal{D}_{v} \to \mathcal{D}_{w}, \nonumber \\
\Sigma_{H}|_{D_{v}} &= \sign(\psi,v) \Sigma_{v,w}, \label{eqn: Kauffman states}
\end{align}  
\end{linenomath*}
where $\Sigma_{v,w}$ denotes the appropriate permutation cobordism. This discussion leads to the following proposition.

\begin{prop}\label{prop: periodic cochain complex}
  Let \(\mathcal{D}\) be an \(m\)-periodic link diagram, then Khovanov bracket \(\khbracket{\mathcal{D}}\) admits an action of \(\z{m}\). 
\end{prop}

\begin{rem}
  This sign convention we use was implicitly described in~\cite{BarNatan}.
\end{rem}

\begin{rem}
  The fact that the action of \(\z{m}\) on the Khovanov bracket is well-defined can be proven using functoriality of Khovanov homology established by~\cite{vogel_functoriality_2015}.
\end{rem}

\begin{proof}[Proof of Proposition~\ref{prop: periodic cochain complex}]
The only thing left to prove, is that $\Sigma_{H}$ commutes with the differential.
Geometric properties of the Khovanov bracket imply that the components of both maps commute up to sign.

Let $x_{1}, \ldots, x_{m} \in V_{\mathcal{T}}$ be homogeneous vectors. Consider the following linear maps
\begin{linenomath*}
\begin{align*}
d_{\mathcal{D}} \colon \left(\Lambda^{\ast} W_{\mathcal{T}}\right)^{\otimes m} &\to \left(\Lambda^{\ast} W_{\mathcal{T}}\right)^{\otimes m}, \\
d_{\mathcal{D}} \colon x_{1} \otimes \ldots \otimes x_{m} &\mapsto \sum_{i=1}^{m}(-1)^{\alpha_{i}}x_{1} \otimes \ldots \otimes d_{\mathcal{T}}(x_{i}) \otimes \ldots \otimes x_{m}, \\
\sigma_{i} \colon \left(\Lambda^{\ast} W_{\mathcal{T}}\right)^{\otimes m} &\to \left(\Lambda^{\ast} W_{\mathcal{T}}\right)^{\otimes m}, \\
\sigma_{i} \colon x_{1} \otimes \ldots \otimes x_{m} &\mapsto (-1)^{\deg x_{i} \cdot \deg x_{i+1}}x_{1} \otimes \ldots \otimes x_{i+1} \otimes x_{i} \otimes \ldots \otimes x_{m},\\
\widetilde{\sigma_{i}} &= (-1)^{n_{-}(\mathcal{T})} \sigma_{i}. 
\end{align*}  
\end{linenomath*}
where $1 \leq i \leq m-1$, $\alpha_{i} = (-1)^{\deg(x_{m}) + \ldots + \deg(x_{i+1})}$ and
\begin{linenomath*}
$$d_{\mathcal{T}}(w) = \sum_{v \in \cross{\mathcal{T}}} w \wedge v.$$  
\end{linenomath*}
The map $\psi$ from (\ref{eqn: permutation map}) is a composition of the maps $\widetilde{\sigma}_{i}$
\begin{linenomath*}
  $$\psi = \widetilde{\sigma}_{m-1} \circ \widetilde{\sigma}_{m-2} \circ \ldots \circ \widetilde{\sigma}_{1}.$$  
\end{linenomath*}
The map $d_{\mathcal{D}}$ corresponds to the differential, i.e. let
\begin{linenomath*}
  $$d_{r-n_{-}(\mathcal{D})} \colon \khbracket{D}^{r - n_{-}(\mathcal{D})} \to \khbracket{\mathcal{D}}^{r+1-n_{-}(\mathcal{D})}$$  
\end{linenomath*}
be the differential in the Khovanov bracket. It is defined
\begin{linenomath*}
  $$d_{r-n_{-}(\mathcal{D})} = \sum_{(v,w)} \sign(v,w) \Sigma_{(v,w)},$$  
\end{linenomath*}
where the summation extends over pairs $(v,w)$ consisting of $v \in \cross{\mathcal{D}}$ and $w$ a is vector from the distinguished basis of $\Lambda^{\ast} W_{D}$, such that $v \wedge w \neq 0$.
Moreover, $\Sigma_{(v,w)}$ is an appropriate elementary cobordism
\begin{linenomath*}
  $$\Sigma_{(v,w)} \colon \mathcal{D}_{w} \to \mathcal{D}_{w \wedge v}.$$  
\end{linenomath*}
It is not hard to check, that the coefficient $\sign(v,w)$ is equal to the coefficient of $v \wedge w$ in $d_{\mathcal{D}}(v)$.
Consequently, in order to finish the proof it is sufficient to check that for $1 \leq i \leq n-1$ the following equality holds
$$\widetilde{\sigma}_{i} \circ d_{D} = d_{D} \circ\widetilde{ \sigma}_{i}.$$
This can be verified by an elementary calculation.
\end{proof}

\begin{figure}
  \centering
  \includegraphics[width = 10cm]{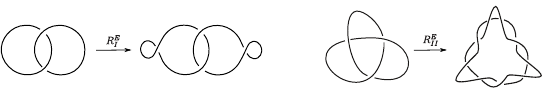}
  \caption{Examples of equivariant Reidemeister moves.}
  \label{fig:equiv-reidemeister-moves}
\end{figure}

Let $\mathcal{D} = \mathcal{D}_{m}(\mathcal{T})$ be an $m$-periodic link diagram obtained from a tangle $\mathcal{T}$. 
Suppose that $\mathcal{T}'$ is tangle obtained from \(\mathcal{T}\) by a single application of one of the Reidemeister moves.
Form another link diagram $\mathcal{D}' = \mathcal{D}_{m}(\mathcal{T}')$. We say that $\mathcal{D}'$ was obtained from $\mathcal{D}$ by an application of an \textit{equivariant Reidemeister move}, see Figure~\ref{fig:equiv-reidemeister-moves}.
Analogously, we can define equivariant planar isotopy for periodic link diagrams.

\begin{prop}\label{prop: equiv reidemeister moves}
  Let $L$ and $L'$ be two $m$-periodic links and let $\mathcal{D}$ and $\mathcal{D}'$ be two $m$-periodic diagrams representing $L$ and $L'$, respectively.
  Every equivariant isotopy from $L$ to $L'$ can be realized by a sequence of equivariant Reidemeister moves from $\mathcal{D}$ to $\mathcal{D}'$.
\end{prop}

\begin{proof}[Proof of Proposition \ref{prop: equiv reidemeister moves}]
If $L$ and $L'$ are equivariantly isotopic, they are also isotopic in the complement of some tubular neighborhood of the fixed point axis. The respective quotient links are also isotopic in the complement of the fixed point axis. The sequence of Reidemeister moves connecting quotient diagrams can be lifted to a sequence of equivariant Reidemeister moves connecting $\mathcal{D}$ and $\mathcal{D}'$.
\end{proof}

\begin{cor}
  Let \(\mathcal{F}\) be any TQFT functor on the category of dotted cobordisms taking values in the category of modules over a commutative ring with unit \(R\).
  If \(\mathcal{D}\) is an \(m\)-periodic link diagram, the \(\mathcal{F}(\khbracket{\mathcal{D}})\) is a complex of \(R[\z{m}]\)-modules.
  In particular, Khovanov complex \(\cc(\mathcal{D};R)\) is a complex of \(R[\z{m}]\)-modules.
\end{cor}

In the nonequivariant setting Bar-Natan~\cite{BarNatan} proved that Reidemeister moves induce chain homotopy equivalences of Khovanov brackets of respective diagrams.
This result can be used to prove invariance of Khovanov bracket of a periodic link up to equivariant Reidemeister moves.

\begin{thm}\label{thm: reidemeister quasi-isomorphism}
Assume that $\mathcal{D}'$ is a periodic link diagrams obtained from a periodic link diagram $\mathcal{D}$ by a single Reidemeister move. Let $\mathcal{F}$ be a TQFT functor whose target is the category of $R$-modules. The map
\begin{linenomath*}
$$\mathcal{R} \colon \khbracket{\mathcal{D}} \to \khbracket{\mathcal{D}'},$$
\end{linenomath*}
induced by an equivariant Reidemeister move gives a chain homotopy equivalence
\begin{linenomath*}
$$\mathcal{F}(\mathcal{R}) \colon \mathcal{F}(\khbracket{\mathcal{D}}) \to \mathcal{F}(\khbracket{\mathcal{D}'})$$
\end{linenomath*}
in the category of cochain complexes of $\GpAAlg{R}{\z{m}}$-modules.
\end{thm}
\begin{proof}
  First let us prove that $\mathcal{F}(\mathcal{R})$ is a morphism in the category of $\GpAAlg{R}{\z{m}}$-modules.
  To verify this condition, refer to the proof of~\cite{BarNatan}*{Theorem 2}. The bracket $\khbracket{\mathcal{D}}$ is constructed along the lines of the formal tensor product of copies of the complex $\khbracket{\mathcal{D}}$. Any collection of morphisms
\begin{linenomath*}
$$\Sigma_{i} \colon \khbracket{\mathcal{T}} \to \khbracket{\mathcal{T}'},$$
\end{linenomath*}
for $i=1, \ldots, m$, yields a morphism
\begin{linenomath*}
$$\mathcal{D}_{m}(\Sigma_{1},\ldots, \Sigma_{m}) \colon \khbracket{\mathcal{D}} \to \khbracket{\mathcal{D}'}.$$
\end{linenomath*}

Taking into account the symmetry of $\mathcal{D}$ and $\mathcal{D}'$, we obtain the following commutative diagram
\begin{center}
\begin{tikzpicture}
\matrix(m)[matrix of math nodes, row sep = 1cm, column sep=2.5cm]
{
  \khbracket{\mathcal{D}} & \khbracket{\mathcal{D'}} \\
  \khbracket{\mathcal{D}} & \khbracket{\mathcal{D'}} \\
};
\path[->, font=\scriptsize]
(m-1-1) edge node[above] {$\mathcal{D}_{m}(\Sigma_{1}, \ldots, \Sigma_{m})$} (m-1-2)
(m-1-1) edge node[left] {$\Sigma_{\mathcal{D}}$} (m-2-1)
(m-1-2) edge node[right] {$\Sigma_{\mathcal{D}'}$} (m-2-2)
(m-2-1) edge node[above] {$\mathcal{D}_{m}(\Sigma_{2}, \ldots, \Sigma_{m}, \Sigma_{1})$} (m-2-2);
\end{tikzpicture}
\end{center}
where $\Sigma_{\mathcal{D}}$ and $\Sigma_{\mathcal{D}'}$ denote the automorphisms of complexes induced by the action of $\z{m}$. Since $\mathcal{R} = \mathcal{D}_{m}(\mathcal{R}', \ldots, \mathcal{R}')$, where
\begin{linenomath*}
$$\mathcal{R}' \colon \khbracket{\mathcal{T}} \to \khbracket{\mathcal{T}'}$$
\end{linenomath*}
is induced by a single Reidemeister move, it follows that $\mathcal{F}(\mathcal{R})$ is a morphism in the category of $\GpAAlg{R}{\z{m}}$-modules.

Repeating the above argument we can prove that chain homotopies defined in Section~4.1 in~\cite{BarNatan} associated to equivariant Reidemeister moves induce chain homotopies in the category of \(R[\z{m}]\)-modules.
\end{proof}

Remainder part of this section is devoted to more in-depth analysis of the action of $\z{m}$ on the Khovanov bracket. The results below will play an important role in Section~\ref{subsec: spectral sequence}.

\begin{defn}\label{defn: Kauffman states of periodic diagram}
Let $\mathcal{S}(\mathcal{D})$ denote the set of smoothings of link diagram $\mathcal{D}$.
\begin{enumerate}
\item For $0 \leq r \leq \#\cross(\mathcal{D})$, let $\mathcal{S}_{r}(\mathcal{D})$ denote the subset of $\mathcal{S}(\mathcal{D})$ consisting of smoothings with exactly $r$ crossings resolved with $1$-smoothing.
\item For $d \mid m$, let $\mathcal{S}^{d}(\mathcal{D})$ denote the set of $d$-periodic smoothings, that is smoothings of the form
  \begin{linenomath*}
  $$\mathcal{D}_{m}(\underbrace{\mathcal{T}_1, \ldots, \mathcal{T}_{\frac{m}{d}}, \mathcal{T}_{1}, \ldots, \mathcal{T}_{\frac{m}{d}}, 
  \ldots, \mathcal{T}_{1}, \ldots, \mathcal{T}_{\frac{m}{d}}}_{\text{$m$ copies}}),$$  
\end{linenomath*}
where $\mathcal{T}_{1}, \ldots, \mathcal{T}_{\frac{m}{d}}$ are distinct smoothings of $\mathcal{T}$, in particular an element of \(\mathcal{S}^d(\mathcal{D})\) is not invariant under any subgroup of \(\z{m}\) which contains properly \(\z{d}\).
\item Define $\mathcal{S}_{r}^{d}(\mathcal{D}) = \mathcal{S}^{d}(\mathcal{D}) \cap \mathcal{S}_{r}(\mathcal{D})$.
\item Define $\overline{\mathcal{S}}_{r}^{d}(\mathcal{D})$ to be the quotient of $\mathcal{S}_{r}^{d}(\mathcal{D})$ by the induced action of $\z{m}$.
\end{enumerate}
\end{defn}

Before proceeding to the next lemma let us some more notation.
Let \(\zzz_{-}\) be a module over the group ring \(\gpRing{\z{m}}\) such that, if $m$ is even, the generator of the cyclic group acts on $\zzz_{-}$ by multiplication by $-1$, trivial module otherwise. Moreover, if \(d \mid m\), and \(M\) is a \(\zzz[\z{d}]\)-module, define the \emph{induction of \(M\)} to be \[\ind{\z{m}}{\z{d}} M = \zzz[\z{m}] \otimes_{\zzz[\z{d}]} M.\]

\begin{lemma}\label{lemma: action on Khovanov complex}
  Suppose that $\mathcal{F}$ is a TQFT functor defined on the category of dotted cobordisms whose target is the category of $R$-modules, for $R$ a commutative ring with unit.
  Let $\mathfrak{s}_{1}, \ldots, \mathfrak{s}_{\frac{m}{d}} \in \mathcal{S}_{r}^{d}(\mathcal{D})$, for $d \mid \gcd(m,r)$ and $d \geq 1$, be an orbit of smoothings, then
\begin{linenomath*}
$$\bigoplus_{i=1}^{m/d} \mathcal{F}(\khbracket{\mathfrak{s}_{i}}) \cong \ind{\z{m}}{\z{d}}
\left( \mathcal{F}(\khbracket{\mathfrak{s}_{1}}) \otimes_{\zzz} \zzz_{-}^{\otimes s(m,r,d)} \right)$$
\end{linenomath*}
as $\GpAAlg{R}{\z{m}}$-modules, where
\begin{linenomath*}
  $$s(m,r,d) = \frac{(m-1)n_{-}(\mathcal{D}) + r(d-1)}{d}.$$
\end{linenomath*}
\end{lemma}
\begin{proof}
  We will prove that if $\mathfrak{s} \in \mathcal{S}_{r}^{d}(\mathcal{D})$, then $\Sigma_{h}^{\frac{m}{d}}(\mathfrak{s}) = (-1)^{s(m,r,d)}\mathfrak{s}$.
  The orbit of $\mathfrak{s}_{1}$ consists of $m/d$ smoothings which are permuted by the action of $\z{m}$.
  Therefore, the lemma will follow from Proposition~\cite{CurtisReiner}*{Prop. 10.5} once we determine the induced action of $\z{d}$ on $\mathcal{F}(\khbracket{\mathfrak{s}_{1}})$.
  Since $\mathcal{F}(\khbracket{\mathfrak{s}_{1}})$ admits an action of $\z{d}$, the induced action will agree with this one up to sign. 

  Smoothing $\mathfrak{s}_{1}$ corresponds to a vector
  \begin{linenomath*}
    $$w = \underbrace{v \otimes v \otimes \ldots \otimes v}_{d},$$
  \end{linenomath*}
  where $v = v_{1} \otimes v_{2} \otimes \ldots \otimes v_{\frac{m}{d}}$ and $v_{1}, \ldots, v_{\frac{m}{d}} \in \Lambda^{\ast}W_{\mathcal{T}}$ belong to the distinguished basis.
  Consequently, according to the adopted sign convention,
  \begin{linenomath*}
    \begin{align*}
      &\psi^{\frac{m}{d}}(w) = (-1)^{k(d-1) + \frac{n_{-}(\mathcal{T})m(m-1)}{d}} w = (-1)^{\frac{r(d-1) + n_{-}(\mathcal{D})(m-1)}{d}} w,
    \end{align*}
  \end{linenomath*}
  where $k = \deg v_{1} + \deg v_{2} + \ldots + \deg v_{\frac{m}{d}}.$
\end{proof}

\begin{cor}\label{cor: action on khovanov complex}
If $\mathcal{F}$ is as in the previous lemma and $0 \leq r \leq n_{+}(\mathcal{D}) + n_{-}(\mathcal{D})$, then
\begin{linenomath*}
\begin{align*}
&\mathcal{\mathcal{F}}(\khbracket{\mathcal{D}}^{r - n_{-}(\mathcal{D})}) = \\
&\bigoplus_{d \mid \gcd(m,r)} \bigoplus_{s \in \overline{\mathcal{S}}_{r}^{d}} \ind{\z{m}}{\z{d}} \left( \mathcal{F}(\khbracket{\mathfrak{s}})
\otimes_{\zzz} \zzz_{-}^{\otimes s(m,r,d)} \right)\{ r + n_{+}(\mathcal{D}) - n_{-}(\mathcal{D})\}.
\end{align*}
\end{linenomath*}
\end{cor}

\begin{rem}
In the above formula there is a small ambiguity. We identified a smoothing of a periodic link diagram with this orbit.
This notational shortcut does not cause any confusion because all smoothings belonging to the same orbit contribute isomorphic summands to $\khbracket{\mathcal{D}}^{r - n_{-}(\mathcal{D})}$.
\end{rem}

\begin{proof}[Proof of Corollary \ref{cor: action on khovanov complex}]
Since for distinct $d_1, d_2 \mid \gcd(m,r)$ the sets $\mathcal{S}_{r}^{d_1}(\mathcal{D})$, $\mathcal{S}_{r}^{d_2}(\mathcal{D})$ are disjoint, the corollary follows readily from Lemma \ref{lemma: action on Khovanov complex}.
\end{proof}

\section{Equivariant Khovanov homology}
\label{sec: equiv kh}

Before proceeding to the main topic of this section let us introduce some notation.
For a doubly-graded module \(M^{\ast,\ast}\) we define another doubly-graded module \(M[k]\{l\}^{\ast,\ast}\), for \(k,l \in \zzz\), to be \(M^{i,j}[k]\{l\} = M^{i-k,j-l}\).

\subsection{Definition and basic properties of equivariant Khovanov homology}
\label{sec:defin-basic-prop-ekh}

Let, as in the previous section, $\mathcal{D}$ be an $m$-periodic link diagram. Suppose also that $R$ is a commutative ring with unit and $M$ is a module over the group ring $\Lambda_m = \GpAAlg{R}{\z{m}}$. By results of the previous section we know that \(\cc(\mathcal{D};R)\) is a complex of \(\Lambda_m\)-modules.
\begin{defn}
  \emph{Equivariant Khovanov homology} of $\mathcal{D}$ is
  \begin{linenomath*}
  $$\kh_{\z{m}}^{\ast,\ast}(\mathcal{D},M) = \Ext{\ast,\ast}{\Lambda_m}{M}{\cc(\mathcal{D},R)}.$$
\end{linenomath*}

  In other words, if we choose a projective resolution $\mathcal{P}_{\ast} \to M$, equivariant Khovanov homology is the cohomology of the following bicomplex
  \begin{linenomath*}
  \begin{align*}
    \ldots &\xrightarrow{d_{\mathcal{P}_{\ast}}} \Hom{}{\Lambda_m}{\mathcal{P}_{i-1}}{\cc(\mathcal{D},R)} \xrightarrow{d_{\mathcal{P}_{\ast}}} \Hom{}{\Lambda_m}{\mathcal{P}_{i}}{\cc(\mathcal{D},R)} \xrightarrow{d_{\mathcal{P}_{\ast}}} \\
           &\xrightarrow{d_{\mathcal{P}_{\ast}}} \Hom{}{\Lambda_m}{\mathcal{P}_{i+1}}{\cc(\mathcal{D},R)} \to \ldots
  \end{align*}
\end{linenomath*}
\end{defn}

In Section~\ref{sec:kh-bracket-periodic-links} we introduced equivariant Reidemeister moves and proved that any equivariant isotopy of periodic links can be realized by a sequence of equivariant Reidemeister moves. The following theorem establishes invariance of equivariant Khovanov homology under equivariant Reidemeister moves.

\begin{thm} \label{thm:invariance_equivariant_homology}
Equivariant Khovanov homology is an invariant of a periodic link, i.e. it is invariant under equivariant Reidemeister moves.
\end{thm}
\begin{proof}
  Theorem~\ref{thm: reidemeister quasi-isomorphism} implies that a single application of an equivariant Reidemeister move to a periodic link diagram, yields a chain homotopy equivalence of the corresponding Khovanov brackets.
  Therefore, an equivariant Reidemeister move induces a chain homotopy equivalence of the respective Khovanov complexes.
  Properties of Ext groups imply that the homotopy equivalence descends to an isomorphism of equivariant Khovanov homology.
\end{proof}

\begin{figure}
  \centering
  \def\svgscale{0.6}
  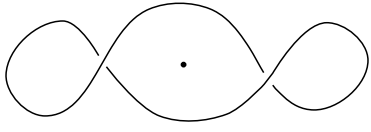
  \caption{$2$-periodic diagram of the unknot and its Khovanov bracket.}
  \label{fig: 2-periodic diagram unknot}
\end{figure}

\begin{rem}
  \label{ex: non homotopy equivalent complexes}
  It is well known that the functor \(\Hom{}{\gpRing{\z{m}}}{M}{-}\), for a \(\gpRing{\z{m}}\)-module \(M\), is in general not exact.
  Lack of exactness of the \(\operatorname{Hom}\) functors forces us to use derived functors of the \(\operatorname{Hom}\) in order to obtain a periodic link invariant. The lack of exactness can be observed for relatively simple examples of links.

  Consider the $2$-periodic diagram $\mathcal{D}$ from Figure~\ref{fig: 2-periodic diagram unknot}. 
  Since
  \begin{linenomath*}
    $$\Hom{}{\gpRing{\z{2}}}{\zzz_{-}}{M} = \{x \in M \colon t \cdot x = -x\},$$
  \end{linenomath*}
  where $t$ denotes the generator of $\z{m}$, we obtain
  \begin{linenomath*}
    \begin{align*}
      &\Hom{}{\gpRing{\z{2}}}{\zzz_{-}}{\cc^{1,\ast}(\mathcal{D})} = \\
      &= \left \langle \Vector{\One \otimes \One}{-\One 
        \otimes \One}, \Vector{\One \otimes X}{-X \otimes \One}, \Vector{X \otimes \One}{- \One \otimes X}, 
        \Vector{X \otimes X}{- X \otimes X} \right \rangle \\
      &\Hom{}{\gpRing{\z{2}}}{\zzz_{-}}{\cc^{2,\ast}(D)} = \left \langle \One, X \right \rangle.
    \end{align*}
  \end{linenomath*}
  Inspection of the differential $d \colon \cc^{1,\ast}(\mathcal{D}) \to \cc^{2,\ast}(\mathcal{D})$ yields
  \begin{linenomath*}
    \begin{align*}
      d \colon \Vector{\One \otimes \One}{- \One \otimes \One} &\mapsto -2 \cdot \One, \\
      d \colon \Vector{\One \otimes X}{- X \otimes \One} &\mapsto -2 \cdot X,
    \end{align*}
  \end{linenomath*}
  therefore
  \begin{linenomath*}
    $$H^{2,\ast}\left( \Hom{}{\gpRing{\z{2}}}{\zzz_{-}}{\cc(\mathcal{D})} \right) = \fff_{2}\{5\} \oplus \fff_{2}\{3\}.$$
  \end{linenomath*}
  On the other hand,
  \begin{linenomath*}
    $$H^{2, \ast}\left( \Hom{}{\gpRing{\z{2}}}{\zzz_{-}}{\cc(\mathcal{U})} \right) = 0,$$
  \end{linenomath*}
  where $\mathcal{U}$ denotes the 2-periodic crossingless unknot diagram.
\end{rem}

Recall that if $C^{\ast}$ and $D^{\ast}$ are bounded chain complexes of $\Lambda_m$-modules and $\mathcal{P}^{\ast} \to C^{\ast}$ is a projective resolution of $C^{\ast}$ and \(\mathcal{Q}^{\ast} \to D^{\ast}\) is a projective resolution of \(D^{\ast}\), then $\Ext{n}{\Lambda_m}{C^{\ast}}{D^{\ast}}$ is isomorphic to the group  of chain homotopy classes of chain maps
\begin{linenomath*}
  $$f \colon \mathcal{P}^{\ast} \to \mathcal{Q}^{\ast}$$
\end{linenomath*}
which increase homological grading by $n$. The composition of chain maps equips $\Gamma(C^{\ast}) = \Ext{\ast}{\Lambda_m}{C^{\ast}}{C^{\ast}}$ with the structure of a graded ring and $\Ext{\ast}{\Lambda_m}{C^{\ast}}{D^{\ast}}$ becomes a graded right module over $\Gamma(C^{\ast})$.
In particular, if we take \(C^{\ast}\) to be the coefficient ring we recover the group cohomology ring \(\Gamma(R) = H^{\ast}(\z{m},R)\).
This leads to the following observation.

\begin{prop}\label{prop:multiplicative-structure}
  Let $\mathcal{D}$ is a periodic link diagram. For any $\Lambda_m$-module $M$, equivariant Khovanov homology $\kh_{\z{m}}^{\ast,\ast}(\mathcal{D},M)$ is a graded module over the ring $\Gamma(M) = \Ext{\ast}{\Lambda_m}{M}{M}$. In particular, $\kh_{\z{m}}(\mathcal{D},R)$ is a graded module over the cohomology ring $H^{\ast}(\z{m},R)$.
\end{prop}

Suppose now that \(\mathcal{D}\) is a periodic diagram of an unlink. Components of \(\mathcal{D}\) can be divided into two categories: components in the first category are invariant under the rotation, component in the second category are freely permuted by \(\z{m}\). We will denote by \(\mathcal{T}_{k,f}\) an \(m\)-periodic unlink diagram consisting of \(k\) free orbits of components and \(f\) fixed components.

\begin{prop}\label{prop:kh-trivial-links}
  For an $m$-periodic unlink diagram \(\mathcal{T}_{k,f}\) and any \(\Lambda_m\)-module \(M\) there exists an isomorphism
  \begin{linenomath*}
    \[\kh_{\z{m}}^{\ast,\ast}(\mathcal{T}_{k,f};M) \cong \Ext{\ast,\ast}{\Lambda_m}{M}{\kh(\mathcal{T}_{k,f};R)}.\]
  \end{linenomath*}
  In particular if \(M = R\) we have
  \begin{linenomath*}
    $$\kh_{\z{m}}^{\ast,\ast}(\mathcal{T}_{k,f}, R) \cong H^{\ast,\ast}(\z{m}, \kh(\mathcal{T}_{k,f},R)).$$
  \end{linenomath*}
  Moreover, if \(k = 0\)
  \[\kh_{\z{m}}^{\ast,\ast}(\mathcal{T}_{0,f}, R) \cong \kh(\mathcal{T}_{0,f};R) \otimes H^{\ast}(\z{m};R).\]
\end{prop}
\begin{proof}
  In order to prove the first claim we can apply Cartan-Eilenberg spectral sequence, see Proposition~5.7.6 in~\cite{Weibel}, converging to equivariant Khovanov homology
  \begin{linenomath*}
    $$E_2^{i,j} = \Ext{i}{\Lambda_m}{M}{\kh^{j,\ast}(\mathcal{D},R)} \Rightarrow \kh_{\z{m}}^{i+j,\ast}(\mathcal{D},M).$$
  \end{linenomath*}
  If $\mathcal{D}$ represents a trivial link, its Khovanov homology is concentrated in homological degree zero, therefore the spectral sequence collapses at the second page and the claim follows.

  Second claim follows directly from the definition of group cohomology, see~Definition~6.1.2 in~\cite{Weibel}. The third claim is a consequence of the fact that \(\kh(\mathcal{T}_{0,f};R)\) is a graded free \(R\)-module with a trivial action of \(\z{m}\).
\end{proof}
When $k > 0$ the action of \(\z{m}\) on $\kh(T_{k,f})$ is nontrivial, but the resulting representation can be decomposed into a direct sum of permutation modules, i.e. modules of the form $\GpAAlg{R}{\z{m}/H}$, for a subgroup $H \subset \z{m}$.
For example, if \(m=p\) is a prime, we have a decomposition $\kh(T_{k,f},R) = \kh(T_{0,f},R) \otimes \kh(T_{k,0},R)$. Moreover, $\kh(T_{k,0},R) = \left(\mathcal{A}^{\otimes p}\right)^{\otimes k}$, where $\mathcal{A} = \GpAAlg{R}{X} / (X^2)$ is the Khovanov algebra and the action on each $\mathcal{A}^{\otimes p}$ factor is given by the formula
\begin{linenomath*}
  $$t \cdot (x_1 \otimes x_2 \otimes \ldots \otimes x_p) = x_p \otimes x_1 \otimes \ldots \otimes x_{p-1}.$$
\end{linenomath*}
For $0 \leq i \leq kp$, of $\mathcal{A}^{\otimes pk}$ we have
\begin{linenomath*}
  \begin{equation}
    \label{eq:decomposition-trivial-link}
    \left(\mathcal{A}^{\otimes kp}\right)^{2i-kp} =
    \begin{cases}
      R^{k \choose i/p} \oplus \GpAAlg{R}{\z{p}}^{\left({pk \choose i} - {k \choose i/p}\right)/p}, & \text{ if } p \mid i, \\
      \GpAAlg{R}{\z{p}}^{{pk \choose i}/p}, & \text{otherwise}.
    \end{cases}
  \end{equation}
\end{linenomath*}

\begin{ex}\label{ex:equiv-kh-trivial-link-T-1-0}
  Consider the case $f = 0$, $k = 1$ and $p$ is an arbitrary prime. Equivariant Khovanov homology of $T_{1,0}$ is given by the following formula
  \begin{linenomath*}
    \begin{equation*}
      \kh_{\z{p}}^{i,2j-p}(T_{1,0},R) =
      \begin{cases}
        H^i(\z{p},R) & \text{if } j \in \{0, p\},  \\
        R^{{p \choose j}/p}, & \text{if } i = 0 \text{ and } 0 < j < p, \\
        0, & \text{otherwise}.
      \end{cases}
    \end{equation*}
  \end{linenomath*}
\end{ex}

\subsection{Integral equivariant Khovanov homology}
\label{sec:integr-equiv-kh}

Let \(L\) be an \(m\)-periodic link.
Let us discuss, as an example, equivariant Khovanov homology over \(R = \zzz\).

First, let us consider, for any \(d \mid m\), the cyclotomic ring \(\cyclotomicRing{d}\), where \(\xi_d = \exp\left(\frac{2 \pi i}{d}\right)\). In other words \(\cyclotomicRing{d} = \zzz[X]/\Phi_{d}(X)\), where \(\Phi_d(X)\) is the \(d\)-th cyclotomic polynomial, see Definition~\ref{defn:cyclotomic_polynomials}.
Since \(\Phi_d(X)\) divides \(X^m-1\), there exists a surjective ring homomorphism \(\Lambda_{m} \to \cyclotomicRing{d}\) which equips the cyclotomic ring with the structure of a \(\Lambda_m\)-module.
Let us define, for any \(d \mid m\),
\begin{linenomath*}
\[\kh_{\z{m}}^{\ast,\ast,d}(L;\zzz) = \kh_{\z{m}}^{\ast,\ast}(L;\cyclotomicRing{d}).\]
\end{linenomath*}

\begin{thm}\label{thm:total_equivariant_homology}
Let $p_{1}, \ldots, p_{s}$ be the collection of all prime divisors of $m$. Define the ring $R_{m} = 
\gpRing{\frac{1}{p_{1}}, \frac{1}{p_{2}}, \ldots, \frac{1}{p_{s}}}.$
There exists a natural map
\begin{linenomath*}
$$\bigoplus_{d | m} \kh_{\z{m}}^{\ast, \ast, d}(L;\zzz) \to \kh(L;\zzz),$$
\end{linenomath*}
which becomes an isomorphism when tensored with $R_{m}$.
\end{thm}
\begin{proof}
  Let us choose an \(m\)-periodic diagram \(\mathcal{D}\) of \(L\). We have
  \begin{linenomath*}
    \begin{equation}\label{eq:free-ext}
      \Ext{\ast,\ast}{\Lambda_m}{\Lambda_m}{\cc(\mathcal{D})} \cong \kh^{\ast,\ast}(\mathcal{D}).
    \end{equation}
  \end{linenomath*}
  Indeed, consider the Cartan-Eilenberg spectral sequence for \(\Ext{\ast,\ast}{\Lambda_m}{\Lambda_m}{\cc(\mathcal{D})}\).
  The \(E_2\) page of this spectral sequence is given by
  \begin{linenomath*}
    \begin{align*}
      &E_{2}^{p,q} = \ext{p,\ast}{\Lambda_m}{\Lambda_m}{\kh^{q,\ast}(\mathcal{D})} = \\
      &= \left \{
        \begin{array}{ll}
          \Hom{}{\Lambda_m}{\Lambda_m}{\kh^{q,\ast}(\mathcal{D})}, & p = 0, \\
          0, & p > 0.\\
        \end{array}
      \right.
    \end{align*}
  \end{linenomath*}
  Therefore, the spectral sequence collapses at the \(E_2\) page and since
  \begin{linenomath*}
    $$\Hom{}{\Lambda_m}{\Lambda_m}{\kh^{q,\ast}(\mathcal{D})) \cong \kh^{q, \ast}(\mathcal{D})},$$
  \end{linenomath*}
  equality~\eqref{eq:free-ext} holds.

  Let \(\Lambda' = \bigoplus_{d \mid m} \cyclotomicRing{d}\).
  Proposition 27.1 from \cite{CurtisReiner} implies that \(\cyclicGroupRing{m} \subset \Lambda'\) and this inclusion induces an isomorphism
  \begin{linenomath*}
    $$\Lambda' \otimes_{\zzz} R_{n} = \Lambda_m \otimes_{\zzz} R_{n} = \GpAAlg{R_{n}}{\z{n}},$$
  \end{linenomath*}
  Consequently 
  \begin{linenomath*}
    \begin{align*}
      &\Ext{r,\ast}{\Lambda_m}{\Lambda'}{\cc(\mathcal{D})} \otimes_{\zzz} R_{n} \cong \\
      &\cong \Ext{r,\ast}{\GpAAlg{R_{n}}{\z{n}}}{\Lambda' \otimes_{\zzz} R_{n}}{\cc(\mathcal{D}) \otimes_{\zzz} 
        R_{n}} \cong \\
      &\cong \Ext{r,\ast}{\GpAAlg{R_{n}}{\z{n}}}{\GpAAlg{R_{n}}{\z{n}}}{\cc(\mathcal{D}) \otimes_{\zzz} R_{n}} 
        \cong \kh^{r, \ast}(\mathcal{D}) \otimes_{\zzz} R_{n},
    \end{align*}
\end{linenomath*}
  because $R_{n}$ is flat over $\zzz$. The last step of the proof consist of noticing that
  \begin{linenomath*}
    $$\Ext{\ast,\ast}{\Lambda_m}{\Lambda'}{\cc(\mathcal{D})} = \bigoplus_{d | r} \kh_{\z{m}}^{\ast, \ast, d}(L).$$
  \end{linenomath*}
\end{proof}

As indicated in Proposition~\ref{prop:multiplicative-structure}, \(\kh_{\z{m}}^{\ast,\ast,d}(L)\) admits an action of
\[\Gamma(\cyclotomicRing{d}) = \Ext{\ast}{\cyclicGroupRing{m}}{\cyclotomicRing{d}}{\cyclotomicRing{d}}.\]
Proposition \ref{prop: properties of ext rings of cylotomic rings} implies that there exists an isomorphism
\begin{linenomath*}
$$\Ext{\ast}{\cyclicGroupRing{m}}{\cyclotomicRing{d}}{\cyclotomicRing{d}} \cong \cyclotomicRing{d}[T_{d}]/(\Psi_{m,d}(\xi_{d})T_{d}),$$
\end{linenomath*}
where $\deg T_{d} = 2$ and \(\Psi_{m,d}(X)\) is a certain integral polynomial, see Definition~\ref{defn:cyclotomic_polynomials}. The action of \(\Gamma(\cyclotomicRing{d})\) imposes some restrictions on algebraic structure of equivariant Khovanov homology.
\begin{prop}\label{prop: periodicity}
Let $\ell \in \zzz$ be chosen such that $\kh^{\ell,\ast}(L) \neq 0$ and $\kh^{i,\ast}(L)$ vanishes for $i > \ell$. Multiplication by $T_{d}$
\begin{linenomath*}
$$- \cup T_{d} \colon \equivKh{m}{i}{\ast}{d}(L) \to \equivKh{m}{i+2}{\ast}{d}(L)$$
\end{linenomath*}
is an epimorphism for $i=\ell$ and isomorphism for $i> \ell$.
\end{prop}

\begin{rem}
This is an analogue of \cite[Lemma 1.1]{Wall}.
\end{rem}

\begin{proof}
Note that Proposition~\ref{prop: properties of ext rings of cylotomic rings} implies that the class \(T_d \in \Gamma^2(\cyclotomicRing{d})\) is represented by the following Yoneda extension
\begin{linenomath*}
  $$0 \to \cyclotomicRing{d} \to \cyclicGroupRing{m} \xrightarrow{\Phi_{d}(t)} \cyclicGroupRing{m} \to \cyclotomicRing{d} \to 0,$$
\end{linenomath*}
where \(t\) denotes a fixed generator of \(\z{m}\).
The above exact sequence can be split into two short exact sequences.
\begin{linenomath*}
  \begin{align}
    0 &\to \cyclotomicRing{d} \to \cyclicGroupRing{m} \to M \to 0 \label{eq:short exact seq 1}\\
    0 &\to M \to \cyclicGroupRing{m} \to \cyclotomicRing{d} \to 0. \label{eq:short exact seq 2}
  \end{align}
\end{linenomath*}
  Let 
\begin{linenomath*}
\begin{align*}
&\delta_{1} \colon \kh_{\z{m}}^{i,\ast,d} \to \Ext{i+1,\ast}{\cyclicGroupRing{m}}{M}{\cc(L)}\\
&\delta_{2} \colon \Ext{i,\ast}{\cyclicGroupRing{m}}{M}{\cc(L)} \to \kh_{\z{m}}^{i+1,\ast,d}(L),
\end{align*}
\end{linenomath*}
be the connecting associated to exact sequences~\eqref{eq:short exact seq 1} and~\eqref{eq:short exact seq 2}, respectively. The map
\begin{linenomath*}
$$-\cup T_{d} \colon \kh_{\z{m}}^{i,\ast,d}(L) \to \kh_{\z{m}}^{i+2,\ast,d}(L)$$
\end{linenomath*}
is equal to the composition $\delta_{2} \circ \delta_{1}$.
Since
\begin{linenomath*}
$$\Ext{i,\ast}{\cyclicGroupRing{m}}{\cyclicGroupRing{m}}{\cc(L)} \cong \kh^{i,\ast}(L),$$
\end{linenomath*}
see the proof of Theorem \ref{thm:total_equivariant_homology}, and $\kh^{i,\ast}(L)=0$ for $i > \ell$ we deduce that $\delta_{1}$ and $\delta_{2}$ 
are epimorphisms for $i = \ell$ and isomorphisms for $i>\ell$.
\end{proof}

\begin{cor}\label{cor: torsion}
Let $\ell$ be as in the previous corollary. For $i > \ell$, $\equivKh{m}{i}{\ast}{d}(L)$ is annihilated by $\frac{m \cdot rad(d)}{d}$, where $rad(d)$ denotes the product of all primes dividing $d$.
\end{cor}
\begin{proof}
This is a direct consequence of the previous corollary and Proposition~\ref{prop: properties of ext rings of cylotomic rings}.
\end{proof}

In Section~\ref{sec:defin-basic-prop-ekh}, we denoted by $\mathcal{T}_{k,f}$ an $m$-periodic unlink diagram consisting of $f$ fixed components and $k$ free orbits.
\begin{ex}\label{ex:computations-trivial-link-1}
  Consider first $\mathcal{T}_{0,f}$. Since all components are fixed, the action of $\z{m}$ on $\kh(\mathcal{T}_{0,f})$ is trivial. Therefore,
  \begin{linenomath*}
    $$\kh_{\z{m}}^{\ast,\ast,d}(\mathcal{T}_{0,f}) = \Ext{\ast}{\gpRing{\z{m}}}{\gpRing{\xi_d}}{\zzz} \otimes \kh(\mathcal{T}_{0,f}).$$
  \end{linenomath*}
  For $d = 1$ we obtain from Proposition~\ref{prop:kh-trivial-links}
  \begin{linenomath*}
    $$\kh_{\z{m}}^{\ast,\ast}(\mathcal{T}_{0,f}) = H^{\ast}(\z{p},\zzz) \otimes_{\zzz} \kh(\mathcal{T}_{0,f}).$$
  \end{linenomath*}
  For the case \(d>1\) Proposition~\ref{prop:ext-rings-d-1} implies that
  \[\kh_{\z{m}}^{i,j,d}(\mathcal{T}_{0,f}) =
    \begin{cases}
      \kh^{0,j}(\mathcal{T}_{0,f};\zzz) \otimes \fff_p, & 2 \nmid i \text{ and } d = p^{\alpha}, \\
      0,                                                & \text {otherwise}.
    \end{cases}\]
\end{ex}

\begin{ex}
  Consider the case $f = 0$, $k = 1$ and $p$ is an arbitrary prime. Equivariant Khovanov homology of $T_{1,0}$ is given by the following formula
\begin{linenomath*}
  \begin{align*}
  \kh_{\z{p}}^{i,2j-p,1}(T_{1,0}) &=
  \begin{cases}
    \fff_{p} & \text{if } j \in \{0, p\} \text{ and } i = 2k, k > 0 \\
    \zzz^{{p \choose j}/p}, & \text{if } i = 0 \text{ and } 0 \leq j \leq p, \\
    0, & \text{otherwise}.
  \end{cases} \\
  \kh_{\z{p}}^{i, 2j-p, p} (T_{1,0}) &=
  \begin{cases}
    \fff_{p} & \text{if } j = 0, p \text{ and } i = 2k+1, k >0,\\
    \gpRing{\xi_p}^{{p \choose j}/p}, & \text{if } 0 < j < p \text{ and } i = 0, \\
    0, & \text{otherwise}
  \end{cases}
  \end{align*}
\end{linenomath*}
which can be obtained from Example~\ref{ex:equiv-kh-trivial-link-T-1-0} and Example~\ref{ex:computations-trivial-link-1}.
\end{ex}

\subsection{Rational equivariant Khovanov homology}
\label{subsec: equiv kh q}

Let us briefly describe properties of equivariant Khovanov homology over the field of rationals. Recall that the rational group algebra \(\cyclicGroupAlg{m}\) is semi-simple and decomposes into a direct sum of simple modules
\begin{linenomath*}
\[\cyclicGroupAlg{m} = \bigoplus_{d \mid m} \cyclotomicField{d}\]
\end{linenomath*}
where, as in the previous section, \(\cyclotomicField{d} = \qqq[X]/\Phi_d(X)\), where \(\Phi_d(X)\) is the \(d\)-th cyclotomic polynomial, see Definition~\ref{defn:cyclotomic_polynomials}.
For \(d \mid m\) let us denote, for \(d \mid m\),
\begin{linenomath*}
\[\kh_{\z{m}}^{\ast,\ast,d}(L;\qqq) = \kh_{\z{m}}^{\ast,\ast}(L;\cyclotomicField{d}).\]
\end{linenomath*}

\begin{prop}\label{prop: structure rational equiv kh}
If $\mathcal{D}$ is an $m$-periodic link diagram, then
\begin{linenomath*}
  $$\cc^{\ast,\ast}(\mathcal{D}; \qqq) \cong \bigoplus_{d \mid n} \cc_{\z{n}}^{\ast,\ast,d}(\mathcal{D}; \qqq),$$
\end{linenomath*}
where
\begin{linenomath*}
  $$\cc_{\z{n}}^{\ast,\ast,d}(\mathcal{D}; \qqq) = 
  \Hom{}{\cyclicGroupAlg{n}}{\cyclotomicField{d}}{\cc^{\ast,\ast}(\mathcal{D}; \qqq)}.$$
\end{linenomath*}
Moreover,
\begin{linenomath*}
  $$\kh_{\z{n}}^{\ast,\ast,d}(\mathcal{D}; \qqq) = H^{\ast,\ast}(\cc_{\z{n}}^{\ast,\ast,d}(\mathcal{D})).$$
\end{linenomath*}
\end{prop}
\begin{proof}
The proposition is a consequence of the Wedderburn decomposition and Schur's Lemma. For more 
details refer to Chapter 1 of~\cite{CurtisReiner}.
\end{proof}

Recall from Section~\ref{sec: preliminaries} that the \emph{Euler’s totient function} is defined
\[\varphi(n) = \# \{1 \leq j \leq n \colon \gcd(j,n)=1\}.\]
Moreover, for any positive integer \(n\), \(\dim_{\qqq}\cyclotomicField{n} = \varphi(n)\).

\begin{cor}\label{cor: vanishing of equiv kh}
Suppose that $\mathcal{D}$ is an $m$-periodic diagram of a link and let $d \mid m$. If for any $i,j$ we have 
$\dim_{\qqq} \kh^{i,j}(\mathcal{D};\qqq) < \varphi(d)$, then $\kh_{\z{n}}^{\ast,\ast,d}(\mathcal{D}; \qqq) = 0$.
\end{cor}
\begin{proof}
Indeed, since $\kh_{\z{n}}^{\ast,\ast,d}(\mathcal{D}; \qqq)$ is a $\cyclotomicField{d}$ vector space, it 
follows readily that $\dim_{\qqq} \kh_{\z{n}}^{\ast,\ast,d}(\mathcal{D}; \qqq)$ is divisible by $\dim_{\qqq} \cyclotomicField{d} = \varphi(d)$.
\end{proof}

\noindent The above corollary can be used to compute the equivariant Khovanov homology in some 
cases.

\begin{cor}
Let $T(n,2)$ be a torus link. Let $d > 2$ be a divisor of $n$. According to Example
\ref{ex: periodic torus links}, $T(n,2)$ is $d$-periodic. Let $d' > 2$ and $d' \mid d$.
$$\kh^{\ast,\ast,d'}_{\z{d}}(T(n,2); \qqq) = 0.$$
\end{cor}
\begin{proof}
Indeed, because according to \cite{Khovanov1}*{Prop. 35} for all $i,j$ we have
$$\dim_{\qqq} \kh^{i,j}(T(n,2); \qqq) \leq 1$$
and $\varphi(d') > 1$ if $d' > 2$.
\end{proof}

\begin{cor}
  Let \(T(n,3)\) denote a torus link. 
  If $\gcd(3,n) = 1$,  
  \[\kh^{\ast,\ast,3}_{\z{3}}(T(n,3);\qqq) = 0.\]
  If $d > 2$ divides $n$, $d' > 2$ and $d' \mid d$, then 
  \[\kh^{\ast,\ast,d'}_{\z{d}}(T(n,3); \qqq) = 0.\]
\end{cor}
\begin{proof}
  Indeed, Theorem 3.1 from~\cite{Turner2} implies that for all $i,j$
  $$\dim_{\qqq}   \kh^{i,j}(T(n,3); \qqq) \le 1,$$
  provided that $\gcd(3,n)=1$.
\end{proof}

\section{The skein spectral sequence}
\label{subsec: spectral sequence}

In order to compute Khovanov homology one usually uses the skein exact sequence.
Unfortunately, if the link in question is periodic, resolution of a single crossing kills the symmetry.
To remedy this it is necessary to choose an orbit of crossings and take into account all possible resolutions of crossings belonging to the orbit.
These data can be organized into a spectral sequence. The construction of thereof is the main goal of this section.
In the non-equivariant case similar spectral sequence was constructed in~\cite{Turner2,EverittTurner}.

\begin{defn}
Let $\alpha \colon \cross(\mathcal{D}) \to \{0,1,x\}$ be a map.
\begin{enumerate}
\item If $i \in \{0,1,x\}$ define $|\alpha|_{i} = \#\alpha^{-1}(i)$. 
\item Define the support of $\alpha$ to be $\supp \alpha = \alpha^{-1}(\{0,1\})$.
\item Define also the following family of maps
  \begin{linenomath*}
  \begin{align*}
\mathcal{B}_k(X) &= \{\alpha \colon \cross(\mathcal{D}) \to \{0,1,x\} \mid \supp \alpha = X, |\alpha|_{1} = k \}.
\end{align*}
  \end{linenomath*}
\item  Denote by $\mathcal{D}_{\alpha}$ the diagram obtained from $\mathcal{D}$ by resolving crossings from $\alpha^{-1}(0)$ by $0$-smoothing and from $\alpha^{-1}(1)$ by $1$-smoothing. 
\end{enumerate}
\end{defn}

By $\ccc(\mathcal{D})$ we will denote the unnormalized Khovanov complex 
\begin{linenomath*}
  $$\ccc(\mathcal{D}) = \cc(\mathcal{D})[n_{-}(\mathcal{D})]\{2n_{-}(\mathcal{D})-n_{+}(\mathcal{D})\},$$
\end{linenomath*}
where $n_{+}(\mathcal{D})$ and $n_{-}(\mathcal{D})$ denote the number of positive and negative crossings of $\mathcal{D}$, respectively.

Fix a crossing $c \in \cross(\mathcal{D})$ and consider three maps
\begin{linenomath*}
$$\alpha_{0},\alpha_{1},\alpha_{x} \colon \cross(\mathcal{D}) \to \{0,1,x\},$$
\end{linenomath*}
which attain different value at $c$, i.e.  $\alpha_{1}(c) = 1$, $\alpha_{x}(c) = x$ and $\alpha_{0}(c) = 0$, and are identical otherwise. These data yield the  following short exact sequence of complexes
\begin{linenomath*}
\begin{equation} \label{eqn:short_exact_seq_res_crossing}
0 \to \ccc(\mathcal{D}_{\alpha_{1}})[1]\{1\} \to \ccc(\mathcal{D}_{\alpha_{x}}) \to \ccc(\mathcal{D}_{\alpha_{0}}) \to 0,
\end{equation}
see~\cite{turner_five_lect}.
\end{linenomath*}
There exists a chain map
\begin{linenomath*}
$$\delta_{c} \colon \ccc(\mathcal{D}_{\alpha_{0}}) \to \ccc(\mathcal{D}_{\alpha_{1}})\{1\},$$
\end{linenomath*}
such that $\ccc(\mathcal{D}_{\alpha_{x}}) = \cone(\delta_{c})$, where $\cone(\delta_{c})$ denotes the algebraic mapping cone of $\delta_{c}$, and (\ref{eqn:short_exact_seq_res_crossing}) is the corresponding short exact sequence of complexes. The map $\delta_{c}$ is obtained as follows. We identify $\ccc(\mathcal{D}_{\alpha_{0}})$ and $\ccc(\mathcal{D}_{\alpha_{1}})$ with submodules of 
$\ccc(\mathcal{D}_{\alpha_{x}})$ ``generated'' by Kauffman states with $c$ resolved with $0$-smoothing or $1$-smoothing, respectively. As a graded module $\ccc(\mathcal{D}_{\alpha_{x}})$ splits as the direct sum
\begin{linenomath*}
$$\ccc(\mathcal{D}_{\alpha_{x}}) = \ccc(\mathcal{D}_{\alpha_{0}}) \oplus \ccc(\mathcal{D}_{\alpha_{1}})\{1\}[1].$$
\end{linenomath*}
with $\ccc(\mathcal{D}_{\alpha_{1}})$ a subcomplex. If $\pi_{1}$ denotes the projection of $\ccc(\mathcal{D}_{\alpha_{x}})$ onto $\ccc(\mathcal{D}_{\alpha_{1}})$ and $i_{0}$ denotes the inclusion of $\ccc(\mathcal{D}_{\alpha_{0}})$, then
\begin{linenomath*}
$$\delta_{c} = \pi_{1} \circ d \circ i_{0}.$$
\end{linenomath*}

When we consider two crossings $c$ and $c'$, we obtain the following bicomplex
\begin{center}
\begin{tikzpicture}
\matrix(m)[matrix of math nodes, column sep=1cm]
{
                      & \ccc(\mathcal{D}_{\alpha_{10}})\{1\} & \\
\ccc(\mathcal{D}_{\alpha_{00}}) & \oplus                     & \ccc(\mathcal{D}_{\alpha_{11}})\{2\} \\
                      & \ccc(\mathcal{D}_{\alpha_{01}})\{1\} & \\
};
\path[->]
(m-2-1) edge node {} (m-1-2)
(m-1-2) edge node {} (m-2-3)
(m-2-1) edge node {} (m-3-2)
(m-3-2) edge node {} (m-2-3);
\end{tikzpicture}
\end{center}
where $\alpha_{00}$, $\alpha_{10}$, $\alpha_{01}$, $\alpha_{11}$ differ only at $c$ or $c'$ and
\begin{linenomath*}
\begin{align*}
\alpha_{00}(c) &= \alpha(c') = 0, \\
\alpha_{10}(c) &= 1, \quad \alpha_{10}(c') = 0, \\
\alpha_{01}(c) &= 0, \quad \alpha_{01}(c') = 1, \\
\alpha_{11}(c) &= \alpha_{11}(c') = 1.
\end{align*}
\end{linenomath*}
The horizontal maps are defined analogously as in the previous case. The total complex of the above bicomplex is equal to the unnormalized Khovanov complex of $\mathcal{D}_{\alpha_{xx}}$, where $\alpha_{xx}$ agrees with $\alpha_{00}$, $\alpha_{10}$, $\alpha_{01}$ and $\alpha_{11}$ outside $c$ and $c'$ and
\begin{linenomath*}
$$\alpha_{xx}(c) = \alpha_{xx}(c') = x.$$
\end{linenomath*}
Continuing this procedure we obtain the following bicomplex
\begin{linenomath*}
$$N^{i,j,k} = \left\{
\begin{array}{ll}
\bigoplus_{\alpha \in \mathcal{B}_{i}(X)} \limits \ccc^{j,k}(\mathcal{D}_{\alpha})\{i\}, & 0 \leq i \leq \# X, \\
0, & \mbox{otherwise}. \\
\end{array}
\right.$$
\end{linenomath*}
The total complex of $N$, denoted by $\Tot(N)^{\ast,\ast}$ and defined by
\begin{linenomath*}
$$\Tot(N)^{i,j} = \bigoplus_{k+l = i} N^{k,l,j},$$
\end{linenomath*}
is equal to $\ccc(\mathcal{D})$.

\begin{defn}
Let $M^{\ast,\ast,\ast}$ be a bicomplex of graded $\zzz$-modules such that
\begin{linenomath*}
$$M^{i,j,k} = \left\{
\begin{array}{ll}
\bigoplus_{\alpha \in \mathcal{B}_{i}(X)} \limits \cc^{j,k}(\mathcal{D}_{\alpha})[c(\mathcal{D}_{\alpha})]\{i + 3 c(\mathcal{D}_{\alpha}) + \# X\}, & 0 \leq i \leq \# X, \\
0, & \mbox{otherwise}, \\
\end{array}
\right.$$
\end{linenomath*}
where $c(\mathcal{D}_{\alpha}) = n_{-}(\mathcal{D}_{\alpha}) - n_{-}(\mathcal{D})$. Vertical differentials
\begin{linenomath*}
$$M^{i,j, \ast} \to M^{i,j+1,\ast}$$
\end{linenomath*}
are sums of $\pm 1$ multiplicities of differentials in the respective Khovanov complexes. Horizontal differentials
\begin{linenomath*}
$$M^{i,j, \ast} \to M^{i+1,j,\ast},$$
\end{linenomath*}
on the other hand, are induced from the appropriate horizontal maps in the  bicomplex $N^{\ast,\ast,\ast}$.
\end{defn}

\begin{prop}
The total complex of $M^{\ast,\ast,\ast}$ is equal to the Khovanov complex $\cc(\mathcal{D})$.
\end{prop}
\begin{proof}
Since the total complex of $N^{\ast,\ast,\ast}$ is equal to $\ccc(\mathcal{D})$ and
\begin{linenomath*}
$$\cc(\mathcal{D}) = \ccc(\mathcal{D})[-n_{-}(\mathcal{D})]\{n_{+}(\mathcal{D}) - 2 n_{-}(\mathcal{D})\},$$
\end{linenomath*}
we only need to check, that the application of the appropriate shift to $N$ results in $M$.
\begin{linenomath*}
\begin{align*}
&N^{i,j,k}[-n_{-}(\mathcal{D})]\{n_{+}(\mathcal{D})-2n_{-}(\mathcal{D})\} = \\
&= \bigoplus_{\alpha \in \mathcal{B}_{i}(X)} \ccc^{j,k}(\mathcal{D}_{\alpha})[-n_{-}(\mathcal{D})]\{i + n_{+}(\mathcal{D}) - 2n_{-}(\mathcal{D})\} = \\
&= \bigoplus_{\alpha \in \mathcal{B}_{i}(X)} \cc^{j,k}(\mathcal{D}_{\alpha})[c(\mathcal{D}_{\alpha})]\{i + 2c(\mathcal{D}_{\alpha}) + n_{+}(\mathcal{D}) - n_{+}(\mathcal{D}_{\alpha})\} = \\
&= \bigoplus_{\alpha \in \mathcal{B}_{i}(X)} \cc^{j,k}(\mathcal{D}_{\alpha})[c(\mathcal{D}_{\alpha})]\{i + 3c(\mathcal{D}_{\alpha}) + \# X\} = M^{i,j,k},
\end{align*}
\end{linenomath*}
because
\begin{linenomath*}
\begin{align*}
&n_{+}(\mathcal{D}) - n_{+}(\mathcal{D}_{\beta}) = \\
&= \# \cross(\mathcal{D}) - n_{-}(\mathcal{D}) - (\# \cross(\mathcal{D}) - \#X - n_{-}(\mathcal{D}_{\beta})) = \\
&= c(D_{\beta}) + \#X.
\end{align*}
\end{linenomath*}
\end{proof}

\begin{defn}\label{defn: filtration}
Let
\begin{linenomath*}
$$\mathcal{F}_{i}(\mathcal{D},X) = \Tot(\bigoplus_{j \geq i} M^{j,\ast,\ast}),$$
\end{linenomath*}
for $0 \leq i \leq \# X$. The family $\{\mathcal{F}_{i}(\mathcal{D},X)\}_{i}$ is a filtration of the Khovanov complex. This filtration is the column filtration of the bicomplex $M^{\ast,\ast,\ast}$, see~Theorem~2.15 in~\cite{McCleary}.
\end{defn}

\begin{thm}[\cite{EverittTurner}]
  \label{thm:construction}
Let $\mathcal{D}$ be a link diagram and let $X \subset \cross(\mathcal{D})$. The pair $(\mathcal{D}, X)$ determines a spectral 
sequence
\begin{linenomath*}
$$\{E^{\ast,\ast,\ast}_r, d_r\}$$
\end{linenomath*}
of graded modules converging to $\kh^{\ast,\ast}(\mathcal{D},R)$ such that
\begin{linenomath*}
$$E_1^{i,j,\ast} = \bigoplus_{\beta \in \mathcal{B}_i(X)} \kh^{j,\ast}(\mathcal{D}_{\beta},R)[c(\mathcal{D}_{\beta})]\{i + 3c(\mathcal{D}_{\beta}) + \#X\},$$
\end{linenomath*}
where $c(\mathcal{D}_{\beta}) = n_{-}(\mathcal{D}_{\beta}) - n_{-}(\mathcal{D})$.
\end{thm}
\begin{proof}
This is the spectral sequence associated to the column filtration of the bicomplex $M$ as in Theorem~2.15 in~\cite{McCleary}.
\end{proof}

Suppose now, that $\mathcal{D}$ is an $m$-periodic link diagram. If $X \subset \cross{\mathcal{D}}$ is invariant under the action of $\z{m}$ then for any $0 \leq k \leq \#X$ there exists an induced action on $\mathcal{B}_k(X)$. Therefore, members of the filtration are invariant under the action of~$\z{m}$.

\begin{prop}\label{prop:invariant_filtration}
If $X \subset \cross{\mathcal{D}}$ is an invariant subset, then for any \(i \in \zzz\), \(\mathcal{F}_{i}(\mathcal{D},X)\) is a chain complex of $\Lambda_m$-modules.
\end{prop}

Let us assume, for the remainder part of this section, that $X$ consists of a single orbit of crossings.

\begin{defn}
  Let $0 \leq i \leq m$ and $d \mid \gcd(m,i)$. Analogously as in Definition~\ref{defn: Kauffman states of periodic diagram} define 
  \begin{linenomath*}
    $$\mathcal{B}_{i}^{d}(X) = \{ \alpha \in \mathcal{B}_{i}(X) \colon \forall_{g \in \z{d} \subset \z{m}} g \cdot \alpha = \alpha \}.$$
  \end{linenomath*}
  Moreover, denote by $\overline{\mathcal{B}}_{i}^{d}(X)$ the quotient of $\mathcal{B}_{i}^{d}(X)$ by $\z{m}$.
\end{defn}

\begin{lemma}
If $\alpha \in \mathcal{B}_{i}^{d}(X)$, then $\mathcal{D}_{\alpha}$ is $\z{d}$-periodic.
\end{lemma}
\begin{proof}
The lemma follows readily, because such diagrams have similar structure as the Kauffman states belonging to $\mathcal{S}_{i}^{d}(\mathcal{D})$, as in Definition \ref{defn: Kauffman states of periodic diagram}.
\end{proof}

\begin{prop}\label{prop:quotients_of_filtration_as_ZG_modules}
For $0 \leq i \leq m$
\begin{linenomath*}
\begin{align*}
&\mathcal{F}_{i}(\mathcal{D},X) / \mathcal{F}_{i+1}(\mathcal{D},X) = \\
&= \bigoplus_{d \mid \gcd(m,i)} \bigoplus_{\alpha \in \overline{\mathcal{B}}_{i}^{d}(X)} \ind{\z{m}}{\z{d}} \left( \cc(\mathcal{D}_{\alpha})[t(\alpha)]\{q(\alpha)\} \otimes_{\zzz} \zzz_{-}^{s(m,r,d)} \right),
\end{align*}
\end{linenomath*}
where \(s(m,r,d)\) is defined in Lemma~\ref{lemma: action on Khovanov complex}. Moreover
\begin{linenomath*}
\begin{align*}
t(\alpha) &= c(\mathcal{D}_{\alpha}) + i, \\
q(\alpha) &= i + 3c(\mathcal{D}_{\alpha}) + m.
\end{align*}
\end{linenomath*}
\end{prop}
\begin{proof}
Proposition is a direct consequence of Lemma~\ref{lemma: action on Khovanov complex}.
\end{proof}

\begin{thm}
  \label{thm-skein-spectral-sequence}
  For any $m$-periodic diagram $\mathcal{D}$ and any $\Lambda_m$-module $M$ there exists a spectral sequence $\{E_{r}^{\ast,\ast}(M), d_r\}$ which converges to $\kh_{\z{m}}(\mathcal{D}, M)$ with
  \begin{linenomath*}
    \begin{align*}
      E_{q}^{i,\ast} = \bigoplus_{d \mid \gcd(m,i)} \bigoplus_{\alpha \in \overline{\mathcal{B}}_{i}^{d}(X)} \kh_{\z{d}}^{\ast, \ast}\left(\mathcal{D}_{\alpha}, (\res{\z{m}}{\z{d}} M) \otimes_{\zzz} \zzz_{-}^{s(m,r,d)}\right)[t(\alpha)]\{q(\alpha)\}.
  \end{align*}
  \end{linenomath*}
\end{thm}
\begin{proof}
    By~\cite{StacksProjectDerived}*{Lemma \StacksTag{015W}{26.14}} we obtain a spectral sequence $\{E_{r}^{\ast,\ast}(M),d_r\}$ converging to $\kh_{\z{m}}^{\ast,\ast}(\mathcal{D},M)$ such that
  \begin{linenomath*}
    $$E_1^{i,j}(M) = \Ext{j}{\Lambda_m}{M}{\mathcal{F}_{i}(X)/\mathcal{F}_{i+1}(X)}.$$
  \end{linenomath*}
  Proposition \ref{prop:quotients_of_filtration_as_ZG_modules} implies that the quotient $\mathcal{F}_{i}(X)/\mathcal{F}_{i+1}(X)$ decomposes into a direct sum of modules of the form $\ind{\z{m}}{\z{d}} \cc(\mathcal{D_{\alpha}})[t(\alpha)]\{q(\alpha)\}$, where $0 \leq k \leq m$ and $\mathcal{D}_{\alpha}$ is $d$-periodic.
  By Shapiro’s Lemma
  \begin{linenomath*}
    \begin{align*}
      &\Ext{\ast}{\Lambda_m}{M}{\ind{\z{m}}{\z{d}} \left(\cc(\mathcal{D_{\alpha}})[t(\alpha)]\{q(\alpha)\} \otimes_{\zzz} \zzz_{-}^{s(m,r,d)}\right)} \cong \\
      &\cong \Ext{\ast}{\Lambda_d}{\res{\z{m}}{\z{d}} M}{\cc(\mathcal{D_{\alpha}})[t(\alpha)]\{q(\alpha)\} \otimes_{\zzz} \zzz_{-}^{s(m,r,d)}} \cong \\
      &\cong \Ext{\ast}{\Lambda_d}{\res{\z{m}}{\z{d}} M \otimes_{\zzz} \zzz_{-}^{s(m,r,d)}}{\cc(\mathcal{D_{\alpha}})[t(\alpha)]\{q(\alpha)\}} \cong \\
      &\cong \kh_{\z{d}}(\mathcal{D},(\res{\z{m}}{\z{d}} M) \otimes_{\zzz} \zzz_{-}^{s(m,r,d)}).
    \end{align*}
  \end{linenomath*}
  The last isomorphism follows from the fact that for any \(\Lambda_m\)-modules \(M\) and \(N\)
  there exists an isomorphism
  \begin{linenomath*}
  \begin{align*}
    \Hom{}{\Lambda_m}{M}{N} &\xrightarrow{\cong} \Hom{}{\Lambda_m}{M \otimes_{\zzz} \zzz_{-}}{N \otimes_{\zzz} \zzz_{-}}, \\
    f &\mapsto f \otimes \operatorname{id}.
  \end{align*}
\end{linenomath*}
\end{proof}

In Section~\ref{sec:integr-equiv-kh} we studied equivariant Khovanov homology with coefficients in cyclotomic rings.
Assume now that \(m = p^n\) is a power of an odd prime.
As a special case of Theorem~\ref{thm-skein-spectral-sequence} we obtain a spectral sequence for \(\kh_{\z{p^n}}^{\ast,\ast,p^{m-s}}(L)\), for any \(0 \leq s \leq m\).

\begin{cor}\label{cor:spectral_sequence_equivariant_homology_odd_period}
  Let $L$ be a $p^{m}$-periodic link, where $p$ is an odd prime, and let $X \subset \cross{\mathcal{D}}$ consists of a single orbit.
  For any \(0 \leq s \leq m\) we have
\begin{linenomath*}
\begin{align*}
E^{0,j}_{1}(\cyclotomicRing{p^{m-s}}) &= \kh_{\z{p^m}}^{j,\ast,p^{m-s}}(\mathcal{D}_{\alpha_0})[c(\mathcal{D}_{\alpha_0})]\{q(\alpha_{0}) \}, \\
E^{p^{m}, j}_{1}(\cyclotomicRing{p^{m-2}}) &= \kh_{\z{p^m}}^{j, \ast, p^{m-s}} (\mathcal{D}_{\alpha_{1}})[c(\mathcal{D}_{\alpha_1})] \{ q(\alpha_{1}) \},\\
E^{i,j}_{1}(\cyclotomicRing{p^{m-s}}) &= \bigoplus_{0 \leq v \leq u_{i}} \bigoplus_{\alpha \in \overline{\mathcal{B}}_{i}^{p^{v}}(X)} \kh_{\z{p^{v}}}^{j, \ast, 
k(v,s)}(\mathcal{D}_{\alpha})[c(\mathcal{D}_{\alpha})]\{q(\alpha)\}^{\ell(v,s)}
\end{align*}
\end{linenomath*}
where $0 < i < p^{m}$. Above $\alpha_{0}$, $\alpha_{1}$ are the unique elements of $\mathcal{B}_0(X)$ and $\mathcal{B}_{p^{m}}(X)$, respectively.
Moreover, $i = p^{u_{i}} g$, where $\gcd(p,g) = 1$ and
\begin{linenomath*}
\begin{align*}
q(\alpha) &= i + 3c(\mathcal{D}_{\alpha}) + p^{m}, \\
k(s,v) &= \left\{
\begin{array}{ll}
1, & v \leq s, \\
p^{v-s}, & v > s, \\
\end{array}
\right. \\
\ell(s,v) &= \left\{
\begin{array}{ll}
\varphi(p^{m-s}), & v \leq s, \\
p^{m-v}, & v > s. \\
\end{array}
\right.
\end{align*}
\end{linenomath*}
\end{cor}
\begin{proof}
  Theorem follows from the following identity
  \[\res{\z{p^m}}{\z{p^v}} \cyclotomicRing{p^{m-s}} =
    \begin{cases}
      \zzz^{\varphi(p^{m-s})},  & v \leq s, \\
      \cyclotomicRing{p^{v-s}}^{p^{m-v}}, & v > s,
    \end{cases}\]
  which is Proposition~2.1.19~in~\cite{politarczyk_wojciech_khovanov_2015}.
\end{proof}

The case \(m = 2\) will be important in Section~\ref{sec:computations}.

\begin{thm}\label{thm:spectral_sequence_equivariant_homology_even_period}
  The $E_{1}$ pages of the respective spectral sequences for $2$-periodic links are given below.
  \begin{linenomath*}
    \begin{align*}
      E_{1}^{0,j}(\zzz) &= \kh_{\z{2}}^{j,\ast,s(\mathcal{D}_{00})}(\mathcal{D}_{00})\{3 c(\mathcal{D}_{00}) + 2\}, \\
      E_{1}^{1,j}(\zzz) &= \kh^{j,\ast}(\mathcal{D}_{01})\{3 c(\mathcal{D}_{01}) + 3\}, \\
      E_{1}^{2,j}(\zzz) &= \kh_{\z{2}}^{j,\ast,3 - s(c(\mathcal{D}_{11}))}(\mathcal{D}_{11})\{3 c(\mathcal{D}_{11}) + 4\}, \\
      E_{1}^{0,j}(\zzz_{-}) &= \kh_{\z{2}}^{j,\ast,3-s(D_{00})}(\mathcal{D}_{00})\{3 c(\mathcal{D}_{00}) + 2\}, \\
      E_{1}^{1,j}(\zzz_{-}) &= \kh^{j,\ast}(\mathcal{D}_{01})\{3 c(\mathcal{D}_{01}) + 3\}, \\
      E_{1}^{2,j}(\zzz_{-}) &= \kh_{\z{2}}^{j,\ast,s(c(\mathcal{D}_{11}))}(\mathcal{D}_{11})\{3 c(\mathcal{D}_{11}) + 4\},
    \end{align*}
  \end{linenomath*}
  where $s \in \{1,2\}$ and
  \begin{linenomath*}
    $$s(\mathcal{D}_{\alpha}) \equiv \frac{c(\mathcal{D}_{\alpha})}{2} \pmod {2}.$$
  \end{linenomath*}
\end{thm}
\begin{proof}
  This theorem follows directly from~\ref{thm-skein-spectral-sequence}.
\end{proof}

\section{Equivariant Khovanov homology of torus links $T(n,2)$}
\label{sec:computations}
In Section~\ref{subsec: equiv kh q} we showed that in some cases simple dimension counting arguments can be used to determine equivariant Khovanov homology.
In particular for torus link \(T(n,2)\) we were able to compute \(\kh_{\z{m}}^{\ast,\ast,d}(T(n,2))\) for any \(m \mid n\) and \(d \mid m\) such that \(d>2\).
In this section we will focus on the case \(m=2\).
In the remainder part of this section we will work with a fixed \(2\)-periodic diagram of $T(m,2)$ obtained by closing the braid diagram depicted in Figure~\ref{fig:T(n,2) link}.

Before we start, let us introduce some notation.
\begin{linenomath*}
  \begin{align}
    \khp(L;\qqq) &= \sum_{i,j}t^{i}q^j \dim_{\qqq} \kh^{i,j}(L;\qqq),\label{eq:khp} \\
    \khp_{1}(L) &= \sum_{i,j} t^i q^j \dim_{\qqq}\kh_{\z{2}}^{i,j,1}(L;\qqq) \label{eq:eqviariant-khp-trivial-rep}, \\
    \khp_{2}(L) &= \sum_{i,j} t^i q^j \dim_{\qqq}\kh_{\z{2}}^{i,j,2}(L;\qqq) \label{eq:equivariant-khp-sign-rep}.
  \end{align}
\end{linenomath*}
We will say that~\eqref{eq:khp} is the \textit{Khovanov polynomial} of $L$ and~\eqref{eq:eqviariant-khp-trivial-rep} together with~\eqref{eq:equivariant-khp-sign-rep} are \textit{equivariant Khovanov polynomials} of $L$.
Notice that Proposition~\ref{prop: structure rational equiv kh} implies that 
\begin{linenomath*}
  \begin{equation}
    \label{eq:equivariant-nonequivariant-kh}
    \khp(L;\qqq) = \khp_{1}(L) + \khp_{2}(L).
  \end{equation}
\end{linenomath*}

The main theorem of this section is stated below.

\begin{thm}\label{thm: computation equiv Kh T_n_2}
Equivariant Khovanov polynomials of the $2$-periodic link $T(m,2)$, \(m>1\), are given below.
\begin{linenomath*}
  \begin{align*}
    \khp_{1}(T(2m+1, 2)) &= \khp(T(2m+1,2);\qqq), \\
    \khp_{2}(T(2m+1, 2)) &= 0, \\
    \khp_{1}(T(2m, 2)) &= \khp(T(2m,2);\qqq) - t^{2m}q^{6m}, \\
    \khp_{2}(T(2m, 2)) &= t^{2m} q^{6m}.
  \end{align*}
\end{linenomath*}
\end{thm}
The proof of Theorem~\ref{thm: computation equiv Kh T_n_2} will occupy the rest of this section.

\begin{figure}
  \centering
  \def\svgscale{0.5}
  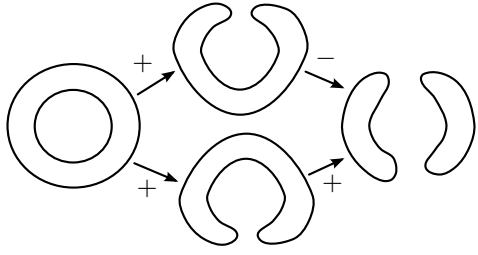
  \caption{Anticommutative cube for $T(2,2)$}
  \label{fig:img_hopf_cube}
\end{figure}

Let us start with detailed computations for the case \(m=2\).
Let $\mathcal{D}_2$ denote the diagram of the Hopf link that can be obtained as the closure of the diagram in Figure~\ref{fig:T(n,2) link}.
Khovanov bracket of \(T(2,2)\) is depicted on Figure~\ref{fig:img_hopf_cube} The action of $\z{2}$ on the Khovanov complex, as discussed in Section~\ref{sec:kh-bracket-periodic-links}, is given by:
  \begin{linenomath*}
    \begin{align*}
      \cc^{0,\ast}(T(2,2))&: \quad a \mapsto a, \\
      \cc^{1,\ast}(T(2,2))&: \quad (a,b) \mapsto (b,a), \\
      \cc^{2,\ast}(T(2,2))&: \quad a \otimes b \mapsto -b \otimes a.
    \end{align*}
  \end{linenomath*}
  Using the above formulas, we can write down a list of generators of the complexes $\cc_{\z{2}}^{\ast,\ast,1}(\mathcal{D},\qqq)$ and $\cc_{\z{2}}^{\ast,\ast,2}(\mathcal{D}, \qqq)$, see Table~\ref{tab: computations kh1 of Hopf link}, and obtain
  \begin{table}
    \centering
    \begin{tabular}{r|c|c|c}
              & $\cc^{0,j,1}(T(2,2))$          & $\cc^{1,j,1}(T(2,2))$     & $\cc^{2,j,1}(T(2,2))$ \\ \hline
      $j = 4$ & $1 \otimes 1$,                & $(1,1)$                  & $1 \otimes X - X \otimes 1$ \\
      $j = 2$ & $1 \otimes X$, $X \otimes 1$  & $(X,X)$                  &    \\
      $j = 0$ & $X \otimes X$                 &                          &    \\ \hline \hline
              & $\cc^{0,j,2}(T(2,2))$          & $\cc^{1,j,2}(T(2,2))$     & $\cc^{2,j,2}(T(2,2))$ \\ \hline
      $j = 6$ &                               &                          & $1 \otimes 1$ \\
      $j = 4$ &                               & $\frac{1}{2}(-1,1)$      & $1 \otimes X + X \otimes 1$ \\
      $j = 2$ &                               & $\frac{1}{2}(-X,X)$      & $X \otimes X$ \\
    \end{tabular}
    \caption{Generators of $\cc_{\z{2}}^{\ast,\ast,1}(T(2,2); \qqq)$ and $\cc_{\z{2}}^{\ast,\ast,2}(T(2,2);\qqq)$.}
    \label{tab: computations kh1 of Hopf link}
  \end{table}
  \begin{linenomath*}
    \begin{align}
      \khp_{1}(T(2,2))(t,q) &= 1 + q^2 + t^2 q^4, \label{eq:equiv-khp-Hopf-link-trivial-rep}\\
      \khp_{2}(T(2,2))(t,q) &= t^2 q^6. \label{eq:equiv-khp-Hopf-link-sign-rep}
    \end{align}
  \end{linenomath*}

Before proceeding further, let us recall that, according to~\cite{Khovanov1}, Khovanov polynomial of the torus link $T(m,2)$ is equal to
\begin{linenomath*}
\begin{align}
\khp(T(2k,2)) &= q^{2k-2} + q^{2k} + t^2 q^{2k+2} (1 + tq^4) \sum_{j = 0}^{k-2} t^{2j} q^{4j} \label{eq:kh-torus-link} \\
              &+ t^{2k} q^{6k - 2} + t^{2k} q^{6k} \nonumber \\
\khp(T(2k+1,2)) &= q^{2k-1} + q^{2k+1} + t^{2}q^{2k+3}(1+tq^{4}) \sum_{j = 0}^{k-1} t^{2j} q^{4j} \label{eq:kh-torus-knot}
\end{align}
\end{linenomath*}
for $k > 0$.

\begin{figure}
  \centering
  \def\svgscale{0.7}
  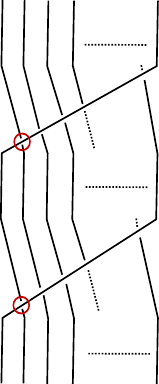
 \caption{The $2$-periodic diagram \(\mathcal{D}_m\) of $T(m,2)$, where \(m\) denotes the number of strands of the depicted braid. The chosen orbit of crossings is marked with red
circles.}
  \label{fig:T(n,2) link}
\end{figure}

\begin{figure}
  \centering
  \def\svgscale{0.5}
  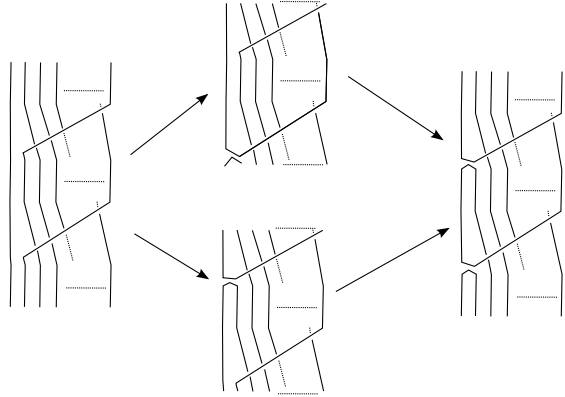
  \caption{Bicomplex associated to the $2$-periodic diagram of $T(m,2)$ from figure \ref{fig:T(n,2)
link}.}
  \label{fig:T(n,2) bicomplex}
\end{figure}

\begin{figure}
  \centering
  \def\svgscale{0.6}
  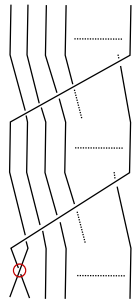
  \caption{Diagram $\mathcal{D}'$ isotopic to the diagram of the $\mathcal{D}_{01}$.}
  \label{fig: T_n_2 D' diag}
\end{figure}


Denote by $\mathcal{D}_{m}$ the $2$-periodic diagram representing the torus link $T(m,2)$ obtained by closing braid diagram from Figure~\ref{fig:T(n,2) link}.
The proof will be inductive with respect to the number of orbits. The base case \(m=2\) was already established. The inductive step will be performed with the aid of the skein spectral sequence. In order to do that choose the orbit of crossings depicted in Figure~\ref{fig:T(n,2) link}.

Let us study the first page of the respective skein spectral sequence.

\begin{lemma}\label{lemma: equiv kh comp 1}
The zeroth and second columns of the $E_{1}$ page of the respective skein spectral sequence applied to the $2$-periodic diagram $\mathcal{D}_m$ is given by
\begin{linenomath*}
  \begin{align*}
    E_{1}^{0,j,k}(\qqq) &= \kh_{\z{2}}^{j,k-1,1}(\mathcal{D}_m;\qqq) \oplus \kh_{\z{2}}^{j,k-3,1}(\mathcal{D}_m;\qqq), \\
    E_{1}^{0,j,k}(\qqq_{-}) &= \kh_{\z{2}}^{j,k-1,2}(\mathcal{D}_m;\qqq) \oplus \kh_{\z{2}}^{j,k-3,2}(\mathcal{D}_m;\qqq), \\
    E_{1}^{2,j,k}(\qqq) &= \kh_{\z{2}}^{j,k-4,2}(\mathcal{D}_{m-2};\qqq), \\
    E_{1}^{2,j,k}(\qqq_{-}) &= \kh_{\z{2}}^{j,k-4,1}(\mathcal{D}_{m-2};\qqq).
  \end{align*}
\end{linenomath*}
\end{lemma}
\begin{proof}
  From Figure~\ref{fig:T(n,2) bicomplex} it is not hard to see, that the diagram $\mathcal{D}_{00}$, i.e. the leftmost diagram, is 
  the split sum $\mathcal{D}_{m-1} \sqcup \mathcal{U}$, where $\mathcal{U}$ is an unknot diagram. Additionally, $\mathcal{D}_{00}$ inherits 
  orientation from $\mathcal{D}_m$, therefore $c(\mathcal{D}_{00})=0$, because $\mathcal{D}_m$ was oriented so that all crossings 
  are positive. This concludes the proof.

  Consider the rightmost diagram $\mathcal{D}_{11}$ in Figure~\ref{fig:T(n,2) bicomplex}. We have $\mathcal{D}_{11} = \mathcal{D}_{m-2}$. It is not hard to check that we can orient $\mathcal{D}_{11}$ in such a way that all crossings are positive, therefore $c(\mathcal{D}_{11})=0.$
\end{proof}

\begin{lemma}\label{lemma: equiv kh comp 3}
The first column of the $E_{1}$ page of the respective skein spectral sequence of the $2$-periodic diagram $\mathcal{D}_m$ is given by
\begin{linenomath*}
  $$E_{1}^{1,j,k}(\qqq) = E_{1}^{1,j,k}(\qqq_{-}) = \left\{
    \begin{array}{ll}
      \kh^{j,k-4}(\mathcal{D}_{2k-2}), & j <  2k-2, \\
      \qqq\{6k-4\}, & j = 2k-2, \\
      \qqq\{6k\}, & j = 2k-1, \\
      0, & j > 2k-1,\\
    \end{array}
  \right.$$
\end{linenomath*}
if $m = 2k$ and
\begin{linenomath*}
  $$E_{1}^{1,j,k}(\qqq) = E_{1}^{1,j,k}(\qqq_{-}) = \left\{
    \begin{array}{ll}
      \kh^{j,k-4}(\mathcal{D}_{2k-1,2}), & j < 2k, \\
      \qqq\{6k+1\} \oplus \qqq\{6k+3\}, & j = 2k, \\
      0, & j > 2k, \\
    \end{array}
  \right.$$
\end{linenomath*}
if $m = 2k+1$.
\end{lemma}
\begin{proof}
Let us denote by $\mathcal{D}_{01}$ one of the diagrams in the middle column of Figure~\ref{fig:T(n,2) bicomplex}. It is not hard to see, that $\mathcal{D}_{01}$ can be oriented in such a way that all crossings are positive, therefore $c(\mathcal{D}_{01})=0$.

Let us denote by $\mathcal{D}'$ the diagram obtained by the closure of the braid from Figure~\ref{fig: T_n_2 D' diag}. Orient it, so that all 
crossings are positive. Notice that $\mathcal{D}_{01}$ and $\mathcal{D}'$ are isotopic. In order to 
prove the lemma, let us compute $\kh(\mathcal{D}';\qqq)$. Let $\mathcal{D}_{0}'$ and $\mathcal{D}_{1}'$ denote $0$-smoothing and $1$-smoothing, respectively, of the crossing marked in Figure~\ref{fig: T_n_2 D' diag}. $\mathcal{D}_{0}'$ is $\mathcal{D}_{m-2}$ with $c(\mathcal{D}_{0}') = 0$.
On the other hand, $\mathcal{D}_{1}'$ represents an unknot and $c(\mathcal{D}_{1}') = m-2$.

Consider first the case $m = 2k$, for $k > 1$. Since $\mathcal{D}_{1}'$ is an unknot diagram there can be only one possibly non-vanishing 
map in the skein exact sequence associated to the crossing marked in Figure~\ref{fig: T_n_2 D' diag}, namely
\begin{linenomath*}
  \begin{equation*}
    \qqq = \kh^{2k-2,6k-6}(\mathcal{D}_{2k-2}) \xrightarrow{\delta} \kh^{0,-1}(\mathcal{D}_1’) = \qqq.
  \end{equation*}
\end{linenomath*}
Suppose that $\delta = 0$ and notice that if $m = 2k$, then $\mathcal{D}'$ represents a knot. It is not 
hard to see, that Khovanov homology of this knot is concentrated only on two diagonals $j-2i = 
2k-3, 2k-1$, regardless of the vanishing of $\delta$. Furthermore, if $\delta$ vanishes then
\begin{linenomath*}
  \begin{equation}
    \khp(\mathcal{D}') = q \khp(\mathcal{D}_{2k-2}) + t^{2k-1}q^{6k-7} + t^{2k-1}q^{6k-5}.
    \label{eqn: khp D'}
  \end{equation}
\end{linenomath*}
On the other hand,~\cite{Lee}*{Theorem 4.4} and~\cite{Rasmussen}*{Proposition 3.3} imply that
\begin{linenomath*}
\begin{equation}
\khp(\mathcal{D}') \equiv q^{s} (q + q^{-1}) \pmod{(1 + tq^{4})},
\label{eqn: khp hthin knot}
\end{equation}
\end{linenomath*}
for some integer $s$. However, from (\ref{eqn: khp D'}) it follows that
\begin{linenomath*}
\begin{align*}
&\khp(\mathcal{D}') \equiv q^{2k-3} + q^{2k-1} + \\
&+ q^{-8k+8}(q^{6k-7} + q^{6k-5}) (1-q^{-4}) \pmod{(1 + tq^{4})},
\end{align*}
\end{linenomath*}
which contradicts \eqref{eqn: khp hthin knot}. Thus, $\delta$ must be nontrivial.

If $m = 2k+1$, there is also only one case to consider, namely
\begin{linenomath*}
$$\qqq = \kh^{2k-1,6k-3}(\mathcal{D}_{2k-1}) \to \kh^{0,-1}(\mathcal{D}_{1}’) = \qqq.$$
\end{linenomath*}
Notice that $\mathcal{D}'$ represents a $2$-component link, whose Khovanov homology is concentrated on two 
diagonals. Once, more by~\cite{Lee}*{Theorem 4.4} and~\cite{Rasmussen}*{Proposition 3.3} we know that
\begin{linenomath*}
\begin{equation}
\khp(\mathcal{D}') \equiv q^{s}(q + q^{-1}) + t^{2} q^{s'}(q+q^{-1}) \pmod{(1+tq^{4})},
\end{equation}
\end{linenomath*}
By argument as in the even case we deduce that $\delta$ must vanish.
\end{proof}

\begin{proof}[Proof of Theorem~\ref{thm: computation equiv Kh T_n_2}]
The proof is inductive. The base case was done in~\eqref{eq:eqviariant-khp-trivial-rep} and~\eqref{eq:equivariant-khp-sign-rep}.

In order to perform the inductive step we will consider two cases.
First, let $m = 2k+1$. From Lemmas~\ref{lemma: equiv kh comp 1} and~\ref{lemma: equiv kh comp 3} we obtain the $E_{1}$ page of the spectral sequence $E_{\ast}^{\ast,\ast,\ast}(\qqq_{-})$, which is depicted in Figure~\ref{fig: E_1 page equiv kh T_2n+1_2}. Since
\begin{linenomath*}
\begin{align*}
\kh^{2k+1,\ast}(\mathcal{D}_{2k+1}) &= \qqq\{6k+3\}, \\
\kh^{2k,\ast}(\mathcal{D}_{2k+1}) &= \qqq\{6k-1\},
\end{align*}
\end{linenomath*}
it follows from~\eqref{eq:kh-torus-link} that the differential
\begin{linenomath*}
$$d_{1}^{0,2k} \colon {}_{2}E_{1}^{0,2k,\ast} \to {}_{2}E_{1}^{1,2k,\ast}$$
\end{linenomath*}
is an isomorphism. Analogous argument using comparison of gradings of $E_{1}^{1,k}(\qqq_{-})$ and 
$\kh^{k+1,\ast}(\mathcal{D}_{2k+1})$ yield that $E_{2}^{\ast,\ast,\ast}(\qqq_{-})$ is zero. Thus, 
\begin{linenomath*}
$$\kh_{\z{2}}^{\ast,\ast,2}(\mathcal{D}_{2k+1};\qqq) = 0,$$
\end{linenomath*}
and consequently by~(\ref{eq:equivariant-nonequivariant-kh}) we obtain
\begin{linenomath*}
$$\kh_{\z{2}}^{\ast,\ast,1}(\mathcal{D}_{2k+1};\qqq) = \kh^{\ast,\ast}(\mathcal{D}_{2k+1}).$$
\end{linenomath*}

Consider now the case $m=2k$. The $E_{1}$ page of the respective skein spectral sequence is presented in 
Figure~\ref{fig: E_1 page equiv kh T_2n+1_2}. Comparison of gradings of 
$E_{1}^{\ast,\ast,\ast}(\qqq_{-})$ and gradings of $\kh^{\ast,\ast}(\mathcal{D}_{2k-2})$ yields that the only 
nonzero entry of $E_{2}^{\ast,\ast,\ast}(\qqq_{-})$ is 
\begin{linenomath*}
  $$E_{2}^{1,2k-1,6k}(\qqq_{-}) = \qqq.$$
\end{linenomath*}
Therefore,
\begin{linenomath*}
  $$\kh_{\z{2}}^{\ast,\ast,2}(\mathcal{D}_{2k};\qqq) = \qqq[2k]\{6k\}.$$
\end{linenomath*}
The desired conclusion follows from~\eqref{eq:equivariant-nonequivariant-kh}.

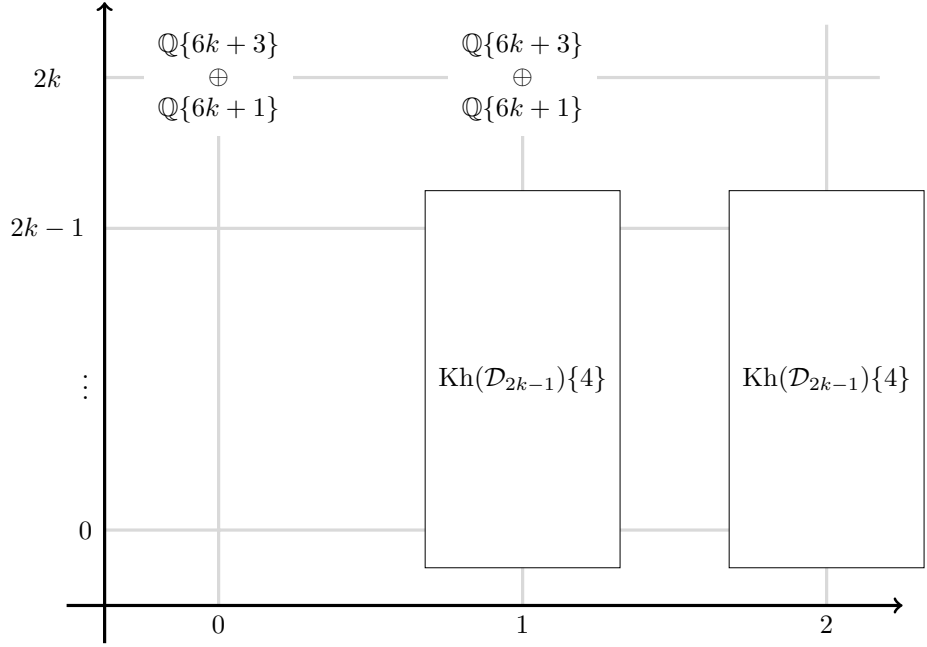
\begin{figure}
\centering
\begin{tikzpicture}
\draw[xstep=4cm, ystep=6cm, gray!30, very thick] (-1.5, -1) grid (8.7, 6.7);
\draw[gray!30, very thick] (-1.5,4) -- (8.7,4);

\draw[->, black, very thick] (-1.5, -1.5) -- (-1.5, 7);
\draw[->, black, very thick] (-2, -1) -- (9, -1);

\node at (0, -1.25) {$0$};
\node at (4, -1.25) {$1$};
\node at (8, -1.25) {$2$};

\node at (-1.75, 0) {$0$};
\node at (-1.75, 2) {$\vdots$};
\node at (-2.25, 4) {$2k-1$};
\node at (-2.25, 6) {$2k$};

\node[fill=white, inner sep=5pt, align=center] at (0, 6) {$\qqq\{6k+3\}$ \\ $\oplus$ \\ 
$\qqq\{6k+1\}$};
\node[fill=white, inner sep=5pt, align=center] at (4, 6) {$\qqq\{6k+3\}$ \\ $\oplus$ \\
$\qqq\{6k+1\}$};
\node[draw, black, very thin, fill=white, inner sep=5pt, minimum height=5cm] at (4,2) 
{$\kh(\mathcal{D}_{2k-1})\{4\}$};
\node[draw, black, very thin, fill=white, inner sep=5pt, minimum height=5cm] at (8,2) 
{$\kh(\mathcal{D}_{2k-1})\{4\}$};
\end{tikzpicture}
\caption{$E_{1}^{\ast,\ast,\ast}(\qqq_{-})$ of $\mathcal{D}_{2k+1}$.}
\label{fig: E_1 page equiv kh T_2n+1_2}
\end{figure}

\begin{figure}
\centering
\begin{tikzpicture}
\draw[xstep=4cm, ystep=6cm, gray!30, very thick] (-1.5, -1) grid (8.7, 6.7);
\draw[gray!30, very thick] (-1.5,5) -- (8.7,5);
\draw[gray!30, very thick] (-1.5,4) -- (8.7,4);

\draw[->, black, very thick] (-1.5, -1.5) -- (-1.5, 7);
\draw[->, black, very thick] (-2, -1) -- (9, -1);

\node at (0, -1.25) {$0$};
\node at (4, -1.25) {$1$};
\node at (8, -1.25) {$2$};

\node at (-1.75, 0) {$0$};
\node at (-1.75, 2) {$\vdots$};
\node at (-2.25, 4) {$2m-3$};
\node at (-2.25, 5) {$2m-2$};
\node at (-2.25, 6) {$2m-1$};

\node[fill=white, inner sep=5pt] at (4, 6) {$\qqq\{6m\}$};
\node[fill=white, inner sep=5pt] at (4, 5) {$\qqq\{6m-4\}$};
\node[fill=white, inner sep=5pt] at (8, 5) {$\qqq\{6m-4\}$};
\node[draw, black, very thin, fill=white, inner sep=5pt, minimum height=5cm] at (4,2) 
{$\kh(T(2m-2,2))\{4\}$};
\node[draw, black, very thin, fill=white, inner sep=5pt, minimum height=5cm] at (8,2) 
{$\kh(T(2m-2,2))\{4\}$};
\end{tikzpicture}
\caption{${}_{2}E_{1}^{\ast,\ast,\ast}$ of $T(2m,2)$.}
\label{fig: E_1 page equiv kh T_2n_2}
\end{figure}
\end{proof}

\appendix

\section{Properties of cyclotomic rings}
\label{sec: preliminaries}

\begin{defn}\label{defn:cyclotomic_polynomials}
Let $m$ be a positive integer, then the \textit{$m$-th cyclotomic polynomial} is the following polynomial
\begin{linenomath*}
$$\Phi_{m}(X) = \prod_{\stackrel{1 \leq a \leq m-1}{\gcd(a,m)=1}} \left(X - \xi_{m}^{a}\right) \in \zzz[X].$$
\end{linenomath*}
Moreover, we \emph{Euler’s totient function} can be defined as
\[\varphi(m) = \deg \Phi_m(X) = \#\{a \colon 1 \leq a \leq m, \quad \gcd(a,m)=1\}.\]
\end{defn}

Let also $d$ be a divisor of $m$. Define the following polynomial
\begin{linenomath*}
$$\Psi_{m,d}(X) = \frac{X^{m}-1}{\Phi_{d}(X)}.$$  
\end{linenomath*}

\begin{prop}\label{prop: properties of ext rings of cylotomic rings}
Let $m$ be a positive integer and $d$ its divisor, then there exists an isomorphism of graded rings
\begin{linenomath*}
$$\ext{\ast}{\gpRing{\z{m}}}{\gpRing{\xi_{d}}}{\gpRing{\xi_{d}}} \cong \gpRing{\xi_{d}}[T_{d}] / (\Psi_{m,d}(\xi_{d}) T_{d}),$$
\end{linenomath*}
where $T_{d} \in \ext{2}{\gpRing{\z{m}}}{\gpRing{\xi_{d}}}{\gpRing{\xi_{d}}}$ is represented by the following Yoneda extension
\begin{linenomath*}
$$0 \to \cyclotomicRing{d} \to \cyclicGroupRing{m} \stackrel{\Phi_{d}(t)}{\longrightarrow} \cyclicGroupRing{m} \to \cyclotomicRing{d} \to 0,$$  
\end{linenomath*}
where $t$ denotes a generator of $\z{m}$. In particular
\begin{linenomath*}
\begin{align*}
\ext{2i}{\cyclicGroupRing{m}}{\cyclotomicRing{d}}{\cyclotomicRing{d}} &= \left \{
\begin{array}{ll}
\z{m}, & i > 0, d=1, \\
A(m,d), & i > 0, d > 0, \\
\cyclotomicRing{d}, & i = 0, \\
\end{array}\right. \\
\ext{2i+1}{\cyclicGroupRing{m}}{\cyclotomicRing{d}}{\cyclotomicRing{d}} &= 0.
\end{align*}  
\end{linenomath*}
where $A(m,d)$ is a finite abelian group of order $\frac{m \cdot rad(d)}{d}$, where $rad(d)$ denotes the product of all primes dividing $d$.
\end{prop}
\begin{proof}
The first isomorphism 
\begin{linenomath*}
$$\ext{\ast}{\gpRing{\z{m}}}{\gpRing{\xi_{d}}}{\gpRing{\xi_{d}}} \cong \gpRing{\xi_{d}}[T_{d}] / (\Psi_{m,d}(\xi_{d}) T_{d}),$$  
\end{linenomath*}
can be verified using \cite{Wall}*{Lemma 1.1} and the following $2$-periodic projective resolution of $\cyclotomicRing{d}$ over $\cyclicGroupRing{m}$
\begin{linenomath*}
$$\ldots \xrightarrow{\Phi_{d}(t)} \cyclicGroupRing{m} \xrightarrow{\Psi_{m,d}(t)} \cyclicGroupRing{m} \xrightarrow{\Phi_{d}(t)} \cyclicGroupRing{m} \xrightarrow{} \cyclotomicRing{d}.$$  
\end{linenomath*}
To prove the second part, notice that if we apply the M\"obius inversion formula, see \cite{IrelandRosen}*{Theorem 2}, to the equality
\begin{linenomath*}
  $$X^{d}-1 = \prod_{d' \mid d} \Phi_{d'}(X)$$
\end{linenomath*}
we obtain
\begin{linenomath*}
$$\Phi_{d}(X) = \prod_{d' \mid d} (X^{d'}-1)^{\mu(d/d')},$$  
\end{linenomath*}
where $\mu$ is the M\"obius function
\begin{linenomath*}
$$\mu(n) = \left\{
  \begin{array}{ll}
    0, & p^{2} \mid n \\
    1, & n = p_{1} \cdot p_{2} \cdot \ldots \cdot p_{2k}, \\
    -1,& n = p_{1} \cdot p_{2} \cdot \ldots \cdot p_{2k+1}, \\ 
  \end{array}
\right.
$$  
\end{linenomath*}
for $k \geq 0$ and $p, p_{i}$ some primes.
Now
\begin{linenomath*}
\begin{align*}
  \Psi_{m,d}(\xi_{d}) &= \lim_{z \to \xi_{d}} \frac{z^{n}-1}{\Phi_{d}(z)}  = \prod_{\stackrel{d' \mid d}{d' < d}} (\xi_{d}^{d/d'}-1)^{-\mu(d')} \lim_{z \to \xi_{d}} \frac{z^{m}-1}{z^{d}-1} = \\
                      &= \frac{m}{d} \prod_{\stackrel{d' \mid d}{d' < d}} (\xi_{d'}-1)^{-\mu(d')}.
\end{align*}  
\end{linenomath*}
Since $\mu(d') = 0$ whenever $d'$ is not square-free, we can restrict the above product to square-free divisors $d'$. Moreover, notice that whenever $d'$ is not prime, then $\xi_{d'}-1$ is invertible, therefore the principal ideal $(\Psi_{m,d}(\xi_{d}))$ can be written as the following product
\begin{linenomath*}
$$\left( \Psi_{m,d}(\xi_{d})\right) = \left(\frac{m}{d}\right) \cdot \prod_{p \mid d} \left (\xi_{p}-1\right).$$  
\end{linenomath*}
Notice also that for any prime $p$ we have $\xi_{p}-1 \mid p$. This implies that the order of the abelian group $\cyclotomicRing{d}/(\Psi_{m,d}(\xi_{d}))$ is equal to $\frac{m \cdot rad(d)}{d}$.
\end{proof}

\begin{prop}
  \label{prop:ext-rings-d-1}
  Let \(m\) and \(d\) be as in the previous proposition. We have
  \begin{linenomath*}
    \begin{align*}
      \Ext{2i+1}{\gpRing{\z{m}}}{\gpRing{\xi_d}}{\zzz} &=
                                                         \begin{cases}
                                                           \fff_{p}, & \text{if } d = p^{\alpha} \text{ is a prime power}, \\
                                                           0,     & \text{otherwise}.
                                                         \end{cases} \\
      \Ext{2i}{\gpRing{\z{m}}}{\gpRing{\xi_d}}{\zzz} &= 0.
    \end{align*}
  \end{linenomath*}
\end{prop}
\begin{proof}
  Consider the following periodic projective resolution
  \begin{linenomath*}
    $$\ldots \xrightarrow{} \gpRing{\z{m}} \xrightarrow{\Psi_{n,d}(t)} \gpRing{\z{m}} \xrightarrow{\Phi_d(t)} \gpRing{\z{m}} \xrightarrow{\Psi_{n,d}(t)} \gpRing{\z{m}} \xrightarrow{\Phi_d(t)} \gpRing{\z{m}}$$
  \end{linenomath*}
  of the cyclotomic module $\gpRing{\xi_d}$. Application of the functor $\Hom{}{\gpRing{\z{m}}}{-}{\zzz}$ to the resolution above yields the following complex
  \begin{linenomath*}
    $$\ldots \xleftarrow{} \zzz \xleftarrow{\Psi_{n,d}(1)} \zzz \xleftarrow{\Phi_d(1)} \zzz \xleftarrow{\Psi_{n,d}(1)} \zzz \xleftarrow{\Phi_d(1)} \zzz.$$
  \end{linenomath*}
  The definition of $\Psi_{n,d}(t)$ implies that it is divisible by \(t-1\), for \(d>1\), hence $\Psi_{n,d}(1) = 0$. Moreover
  \begin{linenomath*}
    $$\Phi_d(1) =
    \begin{cases}
      p, & \text{if } d = p^{\alpha} \text{ is a prime power}, \\
      1, & \text{otherwise}.
    \end{cases}$$
  \end{linenomath*}
\end{proof}

\bibliography{biblio,biblio2}

\end{document}

%% file: crossing_nonoriented.pdf_tex
\begingroup%
  \makeatletter%
  \providecommand\color[2][]{%
    \errmessage{(Inkscape) Color is used for the text in Inkscape, but the package 'color.sty' is not loaded}%
    \renewcommand\color[2][]{}%
  }%
  \providecommand\transparent[1]{%
    \errmessage{(Inkscape) Transparency is used (non-zero) for the text in Inkscape, but the package 'transparent.sty' is not loaded}%
    \renewcommand\transparent[1]{}%
  }%
  \providecommand\rotatebox[2]{#2}%
  \ifx\svgwidth\undefined%
    \setlength{\unitlength}{18.70443764bp}%
    \ifx\svgscale\undefined%
      \relax%
    \else%
      \setlength{\unitlength}{\unitlength * \real{\svgscale}}%
    \fi%
  \else%
    \setlength{\unitlength}{\svgwidth}%
  \fi%
  \global\let\svgwidth\undefined%
  \global\let\svgscale\undefined%
  \makeatother%
  \begin{picture}(1,0.91836368)%
    \put(0,0){\includegraphics[width=\unitlength,page=1]{crossing_nonoriented.pdf}}%
  \end{picture}%
\endgroup%

%% file: 0-smoothing.pdf_tex
\begingroup%
  \makeatletter%
  \providecommand\color[2][]{%
    \errmessage{(Inkscape) Color is used for the text in Inkscape, but the package 'color.sty' is not loaded}%
    \renewcommand\color[2][]{}%
  }%
  \providecommand\transparent[1]{%
    \errmessage{(Inkscape) Transparency is used (non-zero) for the text in Inkscape, but the package 'transparent.sty' is not loaded}%
    \renewcommand\transparent[1]{}%
  }%
  \providecommand\rotatebox[2]{#2}%
  \ifx\svgwidth\undefined%
    \setlength{\unitlength}{18.70443764bp}%
    \ifx\svgscale\undefined%
      \relax%
    \else%
      \setlength{\unitlength}{\unitlength * \real{\svgscale}}%
    \fi%
  \else%
    \setlength{\unitlength}{\svgwidth}%
  \fi%
  \global\let\svgwidth\undefined%
  \global\let\svgscale\undefined%
  \makeatother%
  \begin{picture}(1,0.91836368)%
    \put(0,0){\includegraphics[width=\unitlength,page=1]{0-smoothing.pdf}}%
  \end{picture}%
\endgroup%

%% file: 1-smoothing.pdf_tex
\begingroup%
  \makeatletter%
  \providecommand\color[2][]{%
    \errmessage{(Inkscape) Color is used for the text in Inkscape, but the package 'color.sty' is not loaded}%
    \renewcommand\color[2][]{}%
  }%
  \providecommand\transparent[1]{%
    \errmessage{(Inkscape) Transparency is used (non-zero) for the text in Inkscape, but the package 'transparent.sty' is not loaded}%
    \renewcommand\transparent[1]{}%
  }%
  \providecommand\rotatebox[2]{#2}%
  \ifx\svgwidth\undefined%
    \setlength{\unitlength}{17.17747609bp}%
    \ifx\svgscale\undefined%
      \relax%
    \else%
      \setlength{\unitlength}{\unitlength * \real{\svgscale}}%
    \fi%
  \else%
    \setlength{\unitlength}{\svgwidth}%
  \fi%
  \global\let\svgwidth\undefined%
  \global\let\svgscale\undefined%
  \makeatother%
  \begin{picture}(1,1.08889324)%
    \put(0,0){\includegraphics[width=\unitlength,page=1]{1-smoothing.pdf}}%
  \end{picture}%
\endgroup%

%% file: crossing_pos.pdf_tex
\begingroup%
  \makeatletter%
  \providecommand\color[2][]{%
    \errmessage{(Inkscape) Color is used for the text in Inkscape, but the package 'color.sty' is not loaded}%
    \renewcommand\color[2][]{}%
  }%
  \providecommand\transparent[1]{%
    \errmessage{(Inkscape) Transparency is used (non-zero) for the text in Inkscape, but the package 'transparent.sty' is not loaded}%
    \renewcommand\transparent[1]{}%
  }%
  \providecommand\rotatebox[2]{#2}%
  \ifx\svgwidth\undefined%
    \setlength{\unitlength}{18.70443764bp}%
    \ifx\svgscale\undefined%
      \relax%
    \else%
      \setlength{\unitlength}{\unitlength * \real{\svgscale}}%
    \fi%
  \else%
    \setlength{\unitlength}{\svgwidth}%
  \fi%
  \global\let\svgwidth\undefined%
  \global\let\svgscale\undefined%
  \makeatother%
  \begin{picture}(1,0.91836368)%
    \put(0,0){\includegraphics[width=\unitlength,page=1]{crossing_pos.pdf}}%
  \end{picture}%
\endgroup%

%% file: crossing_neg.pdf_tex
\begingroup%
  \makeatletter%
  \providecommand\color[2][]{%
    \errmessage{(Inkscape) Color is used for the text in Inkscape, but the package 'color.sty' is not loaded}%
    \renewcommand\color[2][]{}%
  }%
  \providecommand\transparent[1]{%
    \errmessage{(Inkscape) Transparency is used (non-zero) for the text in Inkscape, but the package 'transparent.sty' is not loaded}%
    \renewcommand\transparent[1]{}%
  }%
  \providecommand\rotatebox[2]{#2}%
  \ifx\svgwidth\undefined%
    \setlength{\unitlength}{18.70443764bp}%
    \ifx\svgscale\undefined%
      \relax%
    \else%
      \setlength{\unitlength}{\unitlength * \real{\svgscale}}%
    \fi%
  \else%
    \setlength{\unitlength}{\svgwidth}%
  \fi%
  \global\let\svgwidth\undefined%
  \global\let\svgscale\undefined%
  \makeatother%
  \begin{picture}(1,0.91836368)%
    \put(0,0){\includegraphics[width=\unitlength,page=1]{crossing_neg.pdf}}%
  \end{picture}%
\endgroup%

%% file: orient_resolution.pdf_tex
\begingroup%
  \makeatletter%
  \providecommand\color[2][]{%
    \errmessage{(Inkscape) Color is used for the text in Inkscape, but the package 'color.sty' is not loaded}%
    \renewcommand\color[2][]{}%
  }%
  \providecommand\transparent[1]{%
    \errmessage{(Inkscape) Transparency is used (non-zero) for the text in Inkscape, but the package 'transparent.sty' is not loaded}%
    \renewcommand\transparent[1]{}%
  }%
  \providecommand\rotatebox[2]{#2}%
  \ifx\svgwidth\undefined%
    \setlength{\unitlength}{18.70443764bp}%
    \ifx\svgscale\undefined%
      \relax%
    \else%
      \setlength{\unitlength}{\unitlength * \real{\svgscale}}%
    \fi%
  \else%
    \setlength{\unitlength}{\svgwidth}%
  \fi%
  \global\let\svgwidth\undefined%
  \global\let\svgscale\undefined%
  \makeatother%
  \begin{picture}(1,0.91836368)%
    \put(0,0){\includegraphics[width=\unitlength,page=1]{orient_resolution.pdf}}%
  \end{picture}%
\endgroup%

%% file: nonorient_resolution.pdf_tex
\begingroup%
  \makeatletter%
  \providecommand\color[2][]{%
    \errmessage{(Inkscape) Color is used for the text in Inkscape, but the package 'color.sty' is not loaded}%
    \renewcommand\color[2][]{}%
  }%
  \providecommand\transparent[1]{%
    \errmessage{(Inkscape) Transparency is used (non-zero) for the text in Inkscape, but the package 'transparent.sty' is not loaded}%
    \renewcommand\transparent[1]{}%
  }%
  \providecommand\rotatebox[2]{#2}%
  \ifx\svgwidth\undefined%
    \setlength{\unitlength}{18.70443764bp}%
    \ifx\svgscale\undefined%
      \relax%
    \else%
      \setlength{\unitlength}{\unitlength * \real{\svgscale}}%
    \fi%
  \else%
    \setlength{\unitlength}{\svgwidth}%
  \fi%
  \global\let\svgwidth\undefined%
  \global\let\svgscale\undefined%
  \makeatother%
  \begin{picture}(1,0.91836368)%
    \put(0,0){\includegraphics[width=\unitlength,page=1]{nonorient_resolution.pdf}}%
  \end{picture}%
\endgroup%

%% file: splot1.pdf_tex
\begingroup%
  \makeatletter%
  \providecommand\color[2][]{%
    \errmessage{(Inkscape) Color is used for the text in Inkscape, but the package 'color.sty' is not loaded}%
    \renewcommand\color[2][]{}%
  }%
  \providecommand\transparent[1]{%
    \errmessage{(Inkscape) Transparency is used (non-zero) for the text in Inkscape, but the package 'transparent.sty' is not loaded}%
    \renewcommand\transparent[1]{}%
  }%
  \providecommand\rotatebox[2]{#2}%
  \ifx\svgwidth\undefined%
    \setlength{\unitlength}{229.66743614bp}%
    \ifx\svgscale\undefined%
      \relax%
    \else%
      \setlength{\unitlength}{\unitlength * \real{\svgscale}}%
    \fi%
  \else%
    \setlength{\unitlength}{\svgwidth}%
  \fi%
  \global\let\svgwidth\undefined%
  \global\let\svgscale\undefined%
  \makeatother%
  \begin{picture}(1,0.84949118)%
    \put(0,0){\includegraphics[width=\unitlength,page=1]{splot1.pdf}}%
    \put(0.36908666,0.49843242){\color[rgb]{0,0,0}\makebox(0,0)[lt]{\begin{minipage}{0.17665296\unitlength}\centering F\end{minipage}}}%
  \end{picture}%
\endgroup%

%% file: planar_algebra_periodic.pdf_tex
\begingroup%
  \makeatletter%
  \providecommand\color[2][]{%
    \errmessage{(Inkscape) Color is used for the text in Inkscape, but the package 'color.sty' is not loaded}%
    \renewcommand\color[2][]{}%
  }%
  \providecommand\transparent[1]{%
    \errmessage{(Inkscape) Transparency is used (non-zero) for the text in Inkscape, but the package 'transparent.sty' is not loaded}%
    \renewcommand\transparent[1]{}%
  }%
  \providecommand\rotatebox[2]{#2}%
  \ifx\svgwidth\undefined%
    \setlength{\unitlength}{138.17115145bp}%
    \ifx\svgscale\undefined%
      \relax%
    \else%
      \setlength{\unitlength}{\unitlength * \real{\svgscale}}%
    \fi%
  \else%
    \setlength{\unitlength}{\svgwidth}%
  \fi%
  \global\let\svgwidth\undefined%
  \global\let\svgscale\undefined%
  \makeatother%
  \begin{picture}(1,1.02481381)%
    \put(0,0){\includegraphics[width=\unitlength,page=1]{planar_algebra_periodic.pdf}}%
  \end{picture}%
\endgroup%

%% file: planar_algebra_periodic_link.pdf_tex
\begingroup%
  \makeatletter%
  \providecommand\color[2][]{%
    \errmessage{(Inkscape) Color is used for the text in Inkscape, but the package 'color.sty' is not loaded}%
    \renewcommand\color[2][]{}%
  }%
  \providecommand\transparent[1]{%
    \errmessage{(Inkscape) Transparency is used (non-zero) for the text in Inkscape, but the package 'transparent.sty' is not loaded}%
    \renewcommand\transparent[1]{}%
  }%
  \providecommand\rotatebox[2]{#2}%
  \ifx\svgwidth\undefined%
    \setlength{\unitlength}{138.17115145bp}%
    \ifx\svgscale\undefined%
      \relax%
    \else%
      \setlength{\unitlength}{\unitlength * \real{\svgscale}}%
    \fi%
  \else%
    \setlength{\unitlength}{\svgwidth}%
  \fi%
  \global\let\svgwidth\undefined%
  \global\let\svgscale\undefined%
  \makeatother%
  \begin{picture}(1,1.02481381)%
    \put(0,0){\includegraphics[width=\unitlength,page=1]{planar_algebra_periodic_link.pdf}}%
  \end{picture}%
\endgroup%

%% file: 2-periodic-unknot.pdf_tex
\begingroup%
  \makeatletter%
  \providecommand\color[2][]{%
    \errmessage{(Inkscape) Color is used for the text in Inkscape, but the package 'color.sty' is not loaded}%
    \renewcommand\color[2][]{}%
  }%
  \providecommand\transparent[1]{%
    \errmessage{(Inkscape) Transparency is used (non-zero) for the text in Inkscape, but the package 'transparent.sty' is not loaded}%
    \renewcommand\transparent[1]{}%
  }%
  \providecommand\rotatebox[2]{#2}%
  \ifx\svgwidth\undefined%
    \setlength{\unitlength}{225.94371921bp}%
    \ifx\svgscale\undefined%
      \relax%
    \else%
      \setlength{\unitlength}{\unitlength * \real{\svgscale}}%
    \fi%
  \else%
    \setlength{\unitlength}{\svgwidth}%
  \fi%
  \global\let\svgwidth\undefined%
  \global\let\svgscale\undefined%
  \makeatother%
  \begin{picture}(1,0.33052604)%
    \put(0,0){\includegraphics[width=\unitlength,page=1]{2-periodic-unknot.pdf}}%
    \put(0.51661078,-0.12733041){\color[rgb]{0,0,0}\makebox(0,0)[lt]{\begin{minipage}{0.06875925\unitlength}\centering \end{minipage}}}%
  \end{picture}%
\endgroup%

%% file: hopf_ckh.pdf_tex
\begingroup%
  \makeatletter%
  \providecommand\color[2][]{%
    \errmessage{(Inkscape) Color is used for the text in Inkscape, but the package 'color.sty' is not loaded}%
    \renewcommand\color[2][]{}%
  }%
  \providecommand\transparent[1]{%
    \errmessage{(Inkscape) Transparency is used (non-zero) for the text in Inkscape, but the package 'transparent.sty' is not loaded}%
    \renewcommand\transparent[1]{}%
  }%
  \providecommand\rotatebox[2]{#2}%
  \ifx\svgwidth\undefined%
    \setlength{\unitlength}{349.93076878bp}%
    \ifx\svgscale\undefined%
      \relax%
    \else%
      \setlength{\unitlength}{\unitlength * \real{\svgscale}}%
    \fi%
  \else%
    \setlength{\unitlength}{\svgwidth}%
  \fi%
  \global\let\svgwidth\undefined%
  \global\let\svgscale\undefined%
  \makeatother%
  \begin{picture}(1,0.52269757)%
    \put(0,0){\includegraphics[width=\unitlength,page=1]{hopf_ckh.pdf}}%
    \put(-0.03143804,0.55110627){\color[rgb]{0,0,0}\makebox(0,0)[lt]{\begin{minipage}{0.57388015\unitlength}\centering \end{minipage}}}%
    \put(-0.00334435,0.50571741){\color[rgb]{0,0,0}\makebox(0,0)[b]{\smash{}}}%
    \put(-0.00603489,0.73701107){\color[rgb]{0,0,0}\makebox(0,0)[lt]{\begin{minipage}{0.33370496\unitlength}\centering \end{minipage}}}%
    \put(0.01821355,0.8732643){\color[rgb]{0,0,0}\makebox(0,0)[lt]{\begin{minipage}{0.28289866\unitlength}\centering \end{minipage}}}%
    \put(0.29072002,0.37790293){\color[rgb]{0,0,0}\makebox(0,0)[b]{\smash{$+$}}}%
    \put(0.29995755,0.11001518){\color[rgb]{0,0,0}\makebox(0,0)[b]{\smash{$+$}}}%
    \put(0.69486093,0.12040731){\color[rgb]{0,0,0}\makebox(0,0)[b]{\smash{$+$}}}%
    \put(0.68100467,0.38829512){\color[rgb]{0,0,0}\makebox(0,0)[b]{\smash{$-$}}}%
  \end{picture}%
\endgroup%

%% file: T_n_2.pdf_tex
\begingroup%
  \makeatletter%
  \providecommand\color[2][]{%
    \errmessage{(Inkscape) Color is used for the text in Inkscape, but the package 'color.sty' is not loaded}%
    \renewcommand\color[2][]{}%
  }%
  \providecommand\transparent[1]{%
    \errmessage{(Inkscape) Transparency is used (non-zero) for the text in Inkscape, but the package 'transparent.sty' is not loaded}%
    \renewcommand\transparent[1]{}%
  }%
  \providecommand\rotatebox[2]{#2}%
  \ifx\svgwidth\undefined%
    \setlength{\unitlength}{76.21225736bp}%
    \ifx\svgscale\undefined%
      \relax%
    \else%
      \setlength{\unitlength}{\unitlength * \real{\svgscale}}%
    \fi%
  \else%
    \setlength{\unitlength}{\svgwidth}%
  \fi%
  \global\let\svgwidth\undefined%
  \global\let\svgscale\undefined%
  \makeatother%
  \begin{picture}(1,2.41825005)%
    \put(0,0){\includegraphics[width=\unitlength,page=1]{T_n_2.pdf}}%
  \end{picture}%
\endgroup%

%% file: T_n_2_spectral_seq.pdf_tex
\begingroup%
  \makeatletter%
  \providecommand\color[2][]{%
    \errmessage{(Inkscape) Color is used for the text in Inkscape, but the package 'color.sty' is not loaded}%
    \renewcommand\color[2][]{}%
  }%
  \providecommand\transparent[1]{%
    \errmessage{(Inkscape) Transparency is used (non-zero) for the text in Inkscape, but the package 'transparent.sty' is not loaded}%
    \renewcommand\transparent[1]{}%
  }%
  \providecommand\rotatebox[2]{#2}%
  \ifx\svgwidth\undefined%
    \setlength{\unitlength}{412.65775418bp}%
    \ifx\svgscale\undefined%
      \relax%
    \else%
      \setlength{\unitlength}{\unitlength * \real{\svgscale}}%
    \fi%
  \else%
    \setlength{\unitlength}{\svgwidth}%
  \fi%
  \global\let\svgwidth\undefined%
  \global\let\svgscale\undefined%
  \makeatother%
  \begin{picture}(1,0.71669069)%
    \put(0,0){\includegraphics[width=\unitlength,page=1]{T_n_2_spectral_seq.pdf}}%
    \put(0.33816178,0.42697416){\color[rgb]{0,0,0}\makebox(0,0)[lt]{\begin{minipage}{0.35865072\unitlength}\centering \end{minipage}}}%
  \end{picture}%
\endgroup%

%% file: T_n_2_D_01.pdf_tex
\begingroup%
  \makeatletter%
  \providecommand\color[2][]{%
    \errmessage{(Inkscape) Color is used for the text in Inkscape, but the package 'color.sty' is not loaded}%
    \renewcommand\color[2][]{}%
  }%
  \providecommand\transparent[1]{%
    \errmessage{(Inkscape) Transparency is used (non-zero) for the text in Inkscape, but the package 'transparent.sty' is not loaded}%
    \renewcommand\transparent[1]{}%
  }%
  \providecommand\rotatebox[2]{#2}%
  \ifx\svgwidth\undefined%
    \setlength{\unitlength}{77.79097389bp}%
    \ifx\svgscale\undefined%
      \relax%
    \else%
      \setlength{\unitlength}{\unitlength * \real{\svgscale}}%
    \fi%
  \else%
    \setlength{\unitlength}{\svgwidth}%
  \fi%
  \global\let\svgwidth\undefined%
  \global\let\svgscale\undefined%
  \makeatother%
  \begin{picture}(1,2.58608127)%
    \put(0,0){\includegraphics[width=\unitlength,page=1]{T_n_2_D_01.pdf}}%
    \put(-0.31268549,2.67025271){\color[rgb]{0,0,0}\makebox(0,0)[lt]{\begin{minipage}{1.45077424\unitlength}\centering \end{minipage}}}%
    \put(0.48872958,2.47669358){\color[rgb]{0,0,0}\makebox(0,0)[b]{\smash{$n-2$ strands}}}%
  \end{picture}%
\endgroup%

%% file: equiv_kh.bbl
\begin{bibdiv}
\begin{biblist}

\bib{BarNatan}{article}{
      author={Bar-Natan, Dror},
       title={Khovanov's homology for tangles and cobordisms},
        date={2005},
     journal={Geom. Topol.},
      volume={9},
       pages={1443\ndash 1499},
}

\bib{borodzik_khovanov_2017}{article}{
      author={Borodzik, Maciej},
      author={Politarczyk, Wojciech},
       title={Khovanov homology and periodic links},
        date={2017-04},
     journal={arXiv:1704.07316 [math]},
         url={http://arxiv.org/abs/1704.07316},
        note={arXiv: 1704.07316},
}

\bib{Chbili2}{article}{
      author={Chbili, Nafaa},
       title={Equivalent {K}hovanov homology associated with symmetric links},
        date={2010},
     journal={Kobe J. Math.},
      volume={27},
      number={1-2},
       pages={73\ndash 89},
}

\bib{cornish_sutured_2016}{article}{
      author={Cornish, James},
       title={Sutured annular {Khovanov} homology and two periodic braids},
        date={2016-06},
     journal={arXiv:1606.03034 [math]},
         url={http://arxiv.org/abs/1606.03034},
        note={arXiv: 1606.03034},
}

\bib{CurtisReiner}{book}{
      author={Curtis, Charles~W.},
      author={Reiner, Irving},
       title={Methods of representation theory. {V}ol. {I}},
      series={Wiley Classics Library},
   publisher={John Wiley \& Sons, Inc., New York},
        date={1990},
        note={With applications to finite groups and orders, Reprint of the
  1981 original, A Wiley-Interscience Publication},
}

\bib{EverittTurner}{article}{
      author={Everitt, Brent},
      author={Turner, Paul},
       title={Bundles of coloured posets and a {L}eray-{S}erre spectral
  sequence for {K}hovanov homology},
        date={2012},
     journal={Trans. Amer. Math. Soc.},
      volume={364},
      number={6},
       pages={3137\ndash 3158},
}

\bib{Hendricks1}{article}{
      author={Hendricks, Kristen},
       title={Localization of the link {F}loer homology of doubly-periodic
  knots},
        date={2015},
     journal={J. Symplectic Geom.},
      volume={13},
      number={3},
       pages={545\ndash 608},
}

\bib{hendricks_flexible_2016}{article}{
      author={Hendricks, Kristen},
      author={Lipshitz, Robert},
      author={Sarkar, Sucharit},
       title={A flexible construction of equivariant {Floer} homology and
  applications},
    language={en},
        date={2016-12},
        ISSN={1753-8416, 1753-8424},
     journal={Journal of Topology},
      volume={9},
      number={4},
       pages={1153\ndash 1236},
         url={http://doi.wiley.com/10.1112/jtopol/jtw022},
}

\bib{IrelandRosen}{book}{
      author={Ireland, Kenneth},
      author={Rosen, Michael},
       title={A classical introduction to modern number theory},
     edition={Second},
      series={Graduate Texts in Mathematics},
   publisher={Springer-Verlag, New York},
        date={1990},
      volume={84},
}

\bib{Khovanov1}{article}{
      author={Khovanov, Mikhail},
       title={A categorification of the {J}ones polynomial},
        date={2000},
     journal={Duke Math. J.},
      volume={101},
      number={3},
       pages={359\ndash 426},
}

\bib{Lee}{article}{
      author={Lee, Eun~Soo},
       title={An endomorphism of the {K}hovanov invariant},
        date={2005},
     journal={Adv. Math.},
      volume={197},
      number={2},
       pages={554\ndash 586},
}

\bib{BassMorgan}{book}{
      editor={Morgan, John~W.},
      editor={Bass, Hyman},
       title={The {S}mith conjecture},
      series={Pure and Applied Mathematics},
   publisher={Academic Press, Inc., Orlando, FL},
        date={1984},
      volume={112},
        note={Papers presented at the symposium held at Columbia University,
  New York, 1979},
}

\bib{McCleary}{book}{
      author={McCleary, John},
       title={A user's guide to spectral sequences},
     edition={Second},
      series={Cambridge Studies in Advanced Mathematics},
   publisher={Cambridge University Press, Cambridge},
        date={2001},
      volume={58},
}

\bib{Murasugi2}{article}{
      author={Murasugi, Kunio},
       title={On periodic knots},
        date={1971},
     journal={Comment. Math. Helv.},
      volume={46},
       pages={162\ndash 174},
}

\bib{murasugi_jones_1988}{article}{
      author={Murasugi, Kunio},
       title={Jones polynomials of periodic links},
        date={1988},
     journal={Pacific J. Math.},
      volume={131},
      number={2},
       pages={319\ndash 329},
}

\bib{politarczyk_wojciech_khovanov_2015}{misc}{
      author={{Politarczyk, Wojciech}},
       title={Khovanov homology of symmetric links},
     address={Poznań},
        date={2015},
}

\bib{politarczyk_equivariant_2017}{article}{
      author={Politarczyk, Wojciech},
       title={Equivariant {Jones} polynomials of periodic links},
    language={en},
        date={2017-03},
        ISSN={0218-2165, 1793-6527},
     journal={Journal of Knot Theory and Its Ramifications},
      volume={26},
      number={03},
       pages={1741007},
  url={http://www.worldscientific.com/doi/abs/10.1142/S0218216517410073},
}

\bib{Przytycki}{article}{
      author={Przytycki, J\'{o}zef~H.},
       title={On {M}urasugi's and {T}raczyk's criteria for periodic links},
        date={1989},
     journal={Math. Ann.},
      volume={283},
      number={3},
       pages={465\ndash 478},
}

\bib{Rasmussen}{article}{
      author={Rasmussen, Jacob},
       title={Khovanov homology and the slice genus},
        date={2010},
     journal={Invent. Math.},
      volume={182},
      number={2},
       pages={419\ndash 447},
}

\bib{SeidelSmith1}{article}{
      author={Seidel, Paul},
      author={Smith, Ivan},
       title={Localization for involutions in {F}loer cohomology},
        date={2010},
     journal={Geom. Funct. Anal.},
      volume={20},
      number={6},
       pages={1464\ndash 1501},
}

\bib{tomDieck}{book}{
      author={tom Dieck, Tammo},
       title={Transformation groups},
      series={de Gruyter Studies in Mathematics},
   publisher={Walter de Gruyter \& Co., Berlin},
        date={1987},
      volume={8},
}

\bib{Traczyk}{article}{
      author={Traczyk, Pawe{\l}},
       title={{$10_{101}$} has no period {$7$}: a criterion for periodic
  links},
        date={1990},
     journal={Proc. Amer. Math. Soc.},
      volume={108},
      number={3},
       pages={845\ndash 846},
}

\bib{Traczyk3}{article}{
      author={Traczyk, Pawe{\l}},
       title={A criterion for knots of period {$3$}},
        date={1990},
     journal={Topology Appl.},
      volume={36},
      number={3},
       pages={275\ndash 281},
}

\bib{Traczyk2}{article}{
      author={Traczyk, Pawe{\l}},
       title={Periodic knots and the skein polynomial},
        date={1991},
     journal={Invent. Math.},
      volume={106},
      number={1},
       pages={73\ndash 84},
}

\bib{StacksProjectDerived}{misc}{
      author={{T}he {S}tacks~{P}roject {A}uthors},
       title={Stacks {P}roject -- {D}erived {C}ategories},
        date={2014},
         url={http://stacks.math.columbia.edu/download/derived.pdf},
        note={available on-line:
  \url{http://stacks.math.columbia.edu/download/derived.pdf}},
}

\bib{Turner2}{article}{
      author={Turner, Paul},
       title={A spectral sequence for {K}hovanov homology with an application
  to {$(3,q)$}-torus links},
        date={2008},
     journal={Algebr. Geom. Topol.},
      volume={8},
      number={2},
       pages={869\ndash 884},
}

\bib{turner_five_lect}{article}{
      author={Turner, Paul},
       title={Five lectures on {Khovanov} homology},
        date={2017-03},
        ISSN={0218-2165, 1793-6527},
     journal={Journal of Knot Theory and Its Ramifications},
      volume={26},
      number={03},
}

\bib{vogel_functoriality_2015}{article}{
      author={Vogel, Pierre},
       title={Functoriality of {Khovanov} homology},
        date={2015-05},
     journal={arXiv:1505.04545 [math]},
         url={http://arxiv.org/abs/1505.04545},
        note={arXiv: 1505.04545},
}

\bib{Wall}{article}{
      author={Wall, C. T.~C.},
       title={Periodic projective resolutions},
        date={1979},
     journal={Proc. London Math. Soc. (3)},
      volume={39},
      number={3},
       pages={509\ndash 553},
}

\bib{Weibel}{book}{
      author={Weibel, Charles~A.},
       title={An introduction to homological algebra},
      series={Cambridge Studies in Advanced Mathematics},
   publisher={Cambridge University Press, Cambridge},
        date={1994},
      volume={38},
}

\bib{zhang_rank_2017}{article}{
      author={Zhang, Melissa},
       title={A rank inequality for the annular {Khovanov} homology of
  2-periodic links},
        date={2017-07},
     journal={arXiv:1707.03279 [math]},
         url={http://arxiv.org/abs/1707.03279},
        note={arXiv: 1707.03279},
}

\end{biblist}
\end{bibdiv}
